\documentclass[twoside, a4paper,10pt,draft]{amsart}

\usepackage{amsmath}  
\usepackage{amsfonts}
\usepackage{amsthm}  
\usepackage{amssymb}
\usepackage[alphabetic, abbrev]{amsrefs}
\usepackage{enumerate}
\usepackage[all]{xy}
\SelectTips{cm}{}
\SilentMatrices
\usepackage[pdfauthor={Steffen Sagave and Christian Schlichtkrull}]{hyperref}

\numberwithin{equation}{section}
\theoremstyle{plain}
\newtheorem{lemma}[subsection]{Lemma}
\newtheorem{theorem}[subsection]{Theorem}

\newtheorem{corollary}[subsection]{Corollary}
\newtheorem{proposition}[subsection]{Proposition}
\theoremstyle{definition}

\newtheorem{definition}[subsection]{Definition}
\newtheorem{example}[subsection]{Example}
\newtheorem{remark}[subsection]{Remark}

\newcommand{\mF}{{\mathbb F}}

\newcommand{\mN}{{\mathbb N}}

\newcommand{\mR}{{\mathbb R}}
\newcommand{\mS}{{\mathbb S}}
\newcommand{\mZ}{{\mathbb Z}}

\newcommand{\cA}{{\mathcal A}}

\newcommand{\cD}{{\mathcal D}}
\newcommand{\cE}{{\mathcal E}}

\newcommand{\cI}{{\mathcal I}}
\newcommand{\cJ}{{\mathcal J}}

\newcommand{\cK}{{\mathcal K}}
\newcommand{\cL}{\mathcal{L}}

\newcommand{\cN}{{\mathcal N}}
\newcommand{\cO}{{\mathcal O}}

\newcommand{\cU}{{\mathcal U}}
\newcommand{\cV}{{\mathcal V}}
\newcommand{\cW}{{\mathcal W}}

\DeclareMathOperator{\id}{id}

\DeclareMathOperator{\Ev}{Ev}

\DeclareMathOperator{\const}{const}
\DeclareMathOperator{\colim}{colim}
\newcommand{\hocolim}{\operatornamewithlimits{hocolim}}

\DeclareMathOperator{\THH}{THH}

\DeclareMathOperator{\concat}{\sqcup}

\newcommand{\ovl}{\overline}

\newcommand{\ot}{\leftarrow}
\newcommand{\iso}{\cong}

\newcommand{\op}{{\mathrm{op}}}
\newcommand{\sm}{\wedge}

\newcommand{\Gr}{\mathit{Gr}}

\newcommand{\SpO}{\mathit{Sp}^{\! O}}
\newcommand{\Spsym}{\mathit{Sp}^{\Sigma}}
\newcommand{\bld}[1]{{\mathbf{#1}}}

\newcommand{\VintO}{\mathcal V\!\int \! O}
\newcommand{\VintOd}{\mathcal V\!\int \! O(d\oplus-)}
\newcommand{\red}{\widehat}

\newcommand{\cof}{\mathrm{cof}}
\newcommand{\cy}{\mathrm{cy}}
\newcommand{\Det}{\mathrm{Det}}
\newcommand{\even}{\mathrm{ev}}
\newcommand{\gp}{\mathrm{gp}}

\newcommand{\GL}{\mathrm{GL}}
 
\newcommand{\rep}{\mathrm{rep}}
\newcommand{\Top}{\mathit{Top}} 
\newcommand{\TopCat}{\mathit{TopCat}}

\newcommand{\xr}{\xrightarrow}
\newcommand{\xl}{\xleftarrow}

\newcommand{\mD}{\mathbb D}
\newcommand{\Grad}{\mathit{Grad}}
\newcommand{\MO}{\mathit{MO}}
\newcommand{\MOP}{\mathit{MOP}}
\newcommand{\MSO}{\mathit{MSO}}

\newcommand{\MSOPeven}{\mathit{MSOP}_{\mathrm{ev}}}
\newcommand{\MU}{\mathit{MU}}
\newcommand{\MUP}{\mathit{MUP}}

\usepackage[pdftex]{graphics,color} 
\usepackage{eso-pic}
\usepackage{scrtime}

\begin{document}
\title{Virtual vector bundles and graded Thom spectra}

\author{Steffen Sagave} \address{Radboud University Nijmegen, IMAPP, PO Box 9010, 6500 GL Nijmegen,\newline The Netherlands}  \email{s.sagave@math.ru.nl}

\author{Christian Schlichtkrull} \address{Department of Mathematics, University of Bergen, P.O. Box 7803, 5020 Bergen, \newline Norway} \email{christian.schlichtkrull@math.uib.no}

\date{\today}

\begin{abstract}
We introduce a convenient framework for constructing and analyzing orthogonal Thom spectra arising from virtual vector bundles. This framework enables us to set up a theory of orientations and graded Thom isomorphisms with good multiplicative properties. The theory is applied to the analysis of logarithmic structures on commutative ring spectra. 
\end{abstract}   

\maketitle

\section{Introduction}
Classically, the starting point for the construction of a Thom spectrum is a sequence of spaces $X_0\to X_1\to X_2\to\dots$, together with a compatible family of maps $f_n\colon X_n\to BO(n)$. The Thom space $T(f_n)$ of the resulting vector bundle on $X_n$ then constitutes the $n$th space in the corresponding Thom spectrum. From a slightly different point of view, one may start with a space $X$ and a map 
$X\to BO$ to the classifying space for stable vector bundles, then choose a suitable filtration of this map, and proceed with this data as above. This point of view has been developed in great detail by Lewis~\cite[Section~IX]{LMS}. It is also of interest to consider Thom spectra derived from virtual vector bundles in the sense of maps $f\colon X\to BO\times \mZ$ with target the classifying space for $KO$-theory. The naive approach to this is to apply the above procedure over each of the components $BO\times\{n\}$ and then pass to the $n$-fold suspension of the resulting Thom spectrum for each $n$. However, this approach does not reflect the $E_{\infty}$ structure of $BO\times \mZ$ and some care is needed in order to define a graded Thom spectrum functor with good multiplicative properties.

One of the main objectives of the present paper is to introduce a convenient framework for the construction and analysis of graded Thom spectra. Let $\cO$ be the topological category with objects the standard inner product spaces $\mR^n$ for $n\geq 0$, and morphisms the linear isometric isomorphisms. This is a 
permutative topological category under direct sum. We define a topological category $\cW$ by applying Quillen's localization construction \cite{Grayson-higher} to $\cO$, that is, $\cW=\cO^{-1}\cO$. This definition makes $\cW$ a permutative topological category whose classifying space is a model of $BO\times \mZ$. It follows that the set of objects in  $\cW$ has a canonical $\mZ$-grading and we write $\cW_{\{n\}}$ for the $n$th component of $\cW$. 

Writing $\Top$ for the category of compactly generated weak Hausdorff spaces, we define a 
\emph{$\cW$-space} to be a continuous functor from $\cW$ to $\Top$. Let $\Top^{\cW}$ be the functor category of $\cW$-spaces. The permutative structure of $\cW$ induces a symmetric monoidal convolution product on $\Top^{\cW}$, and we use the term \emph{commutative \mbox{$\cW$-space} monoid} for a commutative monoid in $\Top^{\cW}$. Such commutative monoids arise naturally, and in particular we show that the definition of the classical Stiefel manifolds can be upgraded to give a commutative $\cW$-space monoid that we denote by $V$. There is also an oriented Stiefel $\cW$-space 
$\tilde V$, which gives rise to a commutative $\cW$-space monoid $\tilde V_{\even}$ when restricted to even degrees.  

We shall prove that the Stiefel $\cW$-space $V$ defines a graded frame bundle over a suitable graded version of the Grassmannians which we denote by $\Gr$. In order to explain this properly we need the permutative topological category $\cV$ with objects $\mR^n$ for $n\geq 0$, and morphisms the (not necessarily surjective) linear isometries. The graded Grassmannian $\Gr$ is then realized as a commutative monoid in the corresponding symmetric monoidal category of $\cV$-spaces $\Top^{\cV}$, and there is a chain of Quillen equivalences of topological model categories
\begin{equation}\label{eq:intro-Quillen-equivalences}
\Top/\Gr_{h\cV} \simeq \Top^{\cV}/\Gr\simeq \Top^{\cW}/V\simeq \Top^{\cW}
\end{equation}
where $\Gr_{h\cV}$ denotes the homotopy colimit of $\Gr$ over $\cV$ and the weak equivalences in $\Top^{\cV}$ and $\Top^{\cW}$ are detected by the respective homotopy colimit functors. Since $\Gr_{h\cV}$ is a model of $BO\times \mZ$, it follows that $\cW$-spaces represent virtual vector bundles via these equivalences. 

The category $\Top^\cW$ is closely related to the symmetric monoidal category of orthogonal spectra 
$\SpO$ introduced in~\cite{MMSS}: There is an adjunction
\[
\mS^{\cW}\colon \Top^{\cW}\rightleftarrows \SpO\colon \Omega^{\cW}
\]
whose left adjoint $\mS^{\cW}$ is strong symmetric monoidal with respect to the convolution product on 
$\Top^{\cW}$  and the smash product on $\SpO$. Applying $\mS^{\cW}$ to $V$ we get a commutative orthogonal ring spectrum that turns out to be a model of the periodic unoriented cobordism spectrum. Motivated by this we write $\MOP=\mS^{\cW}[V]$.
The graded Thom spectrum functor on $\Top/\Gr_{h\cV}$ is now defined to be the composition of the functors
\begin{equation}\label{eq:intro-graded-Thom-functor}
T\colon \Top/Gr_{h\cV}\to \Top^{\cV}/Gr\to \Top^{\cW}/V\xr{\mS^{\cW}} \SpO/\MOP
\end{equation}
as detailed in Section~\ref{sec:lifting-space-level}.
One may also define a graded Thom spectrum functor on $\Top^{\cV}/\Gr$ by the composition of the two last functors in \eqref{eq:intro-graded-Thom-functor} and we show in Propositions~\ref{prop:T-lax-monoidal} and \ref{prop:V-homotopy-invariance} that this leads to a lax symmetric monoidal homotopy functor that closely resembles the classical construction of Thom spectra from compatible sequences of vector bundles.  

The next theorem (appearing as Theorem~\ref{thm:D-algebra-Thom-functor} in Section~\ref{sec:lifting-space-level}) shows that the graded Thom spectrum functor in \eqref{eq:intro-graded-Thom-functor} has good multiplicative properties. 

\begin{theorem}\label{thm:intro-multiplicative-Thom-functor}
Let $\cD$ be an operad augmented over the Barratt--Eccles operad. Then the graded Thom spectrum functor in \eqref{eq:intro-graded-Thom-functor} induces a functor of $\cD$-algebras
\[
T\colon \Top[\cD]/\Gr_{h\cV} \to \SpO[\cD]/\MOP
\]
\end{theorem}
Here we require $\cD$ to be augmented over the Barratt--Eccles operad in order for $\Gr_{h\cV}$ to inherit the structure of a $\cD$-algebra. This is not a serious restriction since every operad is equivalent to one augmented over the Barratt--Eccles operad.
In the case where $\cD$ is an $E_{\infty}$ operad, we see that the graded Thom spectrum functor takes $E_{\infty}$ spaces over $\Gr_{h\cV}$ to $E_{\infty}$ orthogonal ring spectra over $\MOP$.

\subsection{Orientations and the graded Thom isomorphism}
The theory of orientations also has an attractive formulation in the setting of $\cW$-spaces. For a commutative orthogonal ring spectrum $R$ there is a commutative \mbox{$\cW$-space} monoid $\GL_1^{\cW}(R)$ of \emph{graded $\cW$-space units}. This is the $\cW$-space analogue of the corresponding construction for symmetric spectra introduced by the authors in~\cite{Sagave-S_diagram}. 
In the case of the Eilenberg--Mac Lane spectrum $H\mZ/2$, the $\cW$-space units $\GL_1^{\cW}(H\mZ/2)$ is equivalent to the $0$th part $V_{\{0\}}$ of the Stiefel $\cW$-space, whereas for $H\mZ$, the graded 
$\cW$-space  units $\GL_1^{\cW}(H\mZ)$ is equivalent to the $0$th part $\tilde V_{\{0\}}$ of the oriented Stiefel $\cW$-space. For a general orthogonal ring spectrum $R$, an $R$-orientation of a $\cW$-space is simply encoded by a map to $\GL_1^{\cW}(R)$. Restricted to the degree 0 part, this approach to orientations reduces to the orientation theory developed in \cite{May-ring-spaces} as follows from the fibration sequence
\begin{equation}\label{eq:BO-BGL(R)-sequence}
\GL_1^{\cW}(R)_{h\cW_{\{0\}}}\to BO\to B\GL_1(R),
\end{equation}
where $\GL_1(R)$ denotes the traditional (ungraded) grouplike monoid of units (see Section~\ref{subsec:GL_1W(R)-orientations} for details). For an $E_{\infty}$ ring spectrum, Ando et.\ al.\ \cite{Ando_B_G_H_R-units} construct $A_{\infty}$ and $E_{\infty}$ $R$-algebra Thom spectra associated to spaces over $B\GL_1(R)$. Using the graded units of ring spectra introduced by the authors, it is possible to give a unified treatment that includes both Thom spectra for spaces over 
$B\GL_1(R)$ and the present approach to graded Thom spectra. We plan to return to this in a separate paper.

In the present paper we find it convenient to work with a general notion of orientations relative to a grouplike commutative $\cW$-space monoid $W$. This then leads to the following formulation of the corresponding graded Thom isomorphism: We consider a commutative orthogonal ring spectrum $R$ which is an algebra over the degree 0 part $\mS^{\cW}[W_{\{0\}}]$. In this situation there is an associated commutative periodic ring spectrum $RP\simeq \bigvee_{n\in d\mZ}\Sigma^nR$, where $d$ is a generator for the image of $\pi_0(W_{h\cW})$ in $\pi_0(B\cW)\cong \mZ$.
The following theorem is then obtained by combining Theorems~\ref{thm:R-Thom-iso} and \ref{thm:mult-R-Thom-iso}.

\begin{theorem}\label{thm:intro-Thom-iso}
For a $\cW$-space $X$ equipped with a $W$-orientation, there is a stable equivalence of $R$-module spectra $R\wedge \mS^{\cW}[X]\simeq RP\bigtriangleup X_{h\cW}$. This is a stable equivalence of $E_{\infty}$ $R$-algebras provided that $X$ is a commutative $\cW$-space monoid equipped with a multiplicative orientation.  
\end{theorem}
Here $RP\bigtriangleup X_{h\cW}$ denotes an $R$-module (respectively an $E_{\infty}$ $R$-algebra) whose underlying orthogonal spectrum is stably equivalent to 
$\bigvee_{n\in d\mZ}\Sigma^nR\wedge(X_{h\cW_{\{n\}}+})$. More generally, we establish a multiplicative version of the graded Thom isomorphism for any operad augmented over the Barratt--Eccles operad (see Section~\ref{subsec:mult-orientations}).

The theorem leads to the following identification of the $R$-homology of $\mS^{\cW}[X]$. 
\begin{corollary}
For a commutative $\cW$-space monoid $X$ equipped with a multiplicative $W$-orientation, there is an isomorphism of graded rings
\[
R_*(\mS^{\cW}[X])\cong \bigoplus_{n\in d\mZ}\Sigma^nR_*(X_{h\cW_{\{n\}}}).
\]
\end{corollary}

As a special case of this, a multiplicative orientation of a commutative $\cW$-space monoid $X$ with respect to the evenly graded oriented Stiefel $\cW$-space $\tilde V_{\even}$ gives rise to an evenly graded Thom isomorphism 
\[
H_*(\mS^{\cW}[X],\mZ)\cong \bigoplus_{n\in \mZ} \Sigma^{2n}H_*(X_{h\cW_{\{2n\}}},\mZ).
\]
Having developed the orientation theory for $\cW$-spaces, we extend this to spaces over $\Gr_{h\cV}$ via the Quillen equivalences in \eqref{eq:intro-Quillen-equivalences}. Combining Theorems~\ref{thm:intro-multiplicative-Thom-functor} and~\ref{thm:intro-Thom-iso} we then get multiplicative Thom isomorphisms for graded Thom spectra (see Section~\ref{subsec:space-level-Thom-iso} for details).

\subsection{Applications to logarithmic structures on commutative ring spectra}
Mimicking the analogous notion in the setting of symmetric spectra~ \cite{Sagave-S_diagram,RSS_LogTHH-I}, a pre-log structure on a commutative orthogonal ring spectrum $R$ is given by a commutative $\cW$-space monoid $M$ and a map of commutative $\cW$-space monoids $M\to\Omega^{\cW}(R)$. There is a corresponding notion of a log structure and we prove in 
Proposition~\ref{prop:pre-log-orientation} that a pre-log structure as above defines a $\GL_1^{\cW}(R)$-orientation of $M$ if and only if the associated log structure is trivial. This relates pre-log structures to orientation theory.

Our results in the present paper are used in joint work with Rognes~\cite{RSS_LogTHH-I, RSS_LogTHH-II} on topological logarithmic Hochschild homology (in short log THH). We discuss these applications in Section~8 which also contains an introduction to this circle of ideas. 
One of the features of log THH is that it allows one to set up homotopy cofiber sequences that resemble the localization sequences known from algebraic \mbox{K-theory}. We explain in Section~\ref{subsec:logTHH} how our work leads to new log THH localization sequences such as the homotopy cofiber sequence 
\[
\THH(\MU) \to \THH(\MUP_{\geq0})\to \THH(\MUP_{\geq0},V^{\cU}_{\geq 0})
\]
considered in Example~\ref{ex:logTHHMUP-sequence}

\subsection{Organization} In Section~\ref{sec:cat-of-linear-isometries} we review results about topological categories and introduce the categories $\cV$ and $\cW$. Section~\ref{sec:W-homotopy-theory} contains results about the homotopy theory of $\cW$-spaces, and Section~\ref{sec:W-spaces-SpO} is about the interplay of $\cW$-spaces and orthogonal spectra. In Section~\ref{sec:graded-Grass} we explain how Grassmannians and Stiefel manifolds fit into our setup of diagram spaces and show how they give rise to periodic cobordism spectra.  We set up the graded Thom spectrum functor in Section~\ref{sec:graded-Thom-functor}, and in Section~\ref{sec:graded-Thom-iso} we establish the graded Thom isomorphism in Theorem~\ref{thm:intro-Thom-iso}. Section~\ref{sec:log-structures} contains the applications of our results to the analysis of topological logarithmic structures. In Appendix~\ref{app-sec:SW-goodness} we establish homotopy invariance properties of the functor $\mS^{\cW}$. 

\subsection{Acknowledgments} The authors would like to thank the
  referee for useful comments on the paper. 

\section{Categories of linear isometries}\label{sec:cat-of-linear-isometries}
We first review some general material on topological categories before turning to the relevant categories of linear isometries.

\subsection{Topological categories}\label{subsec:Top-categories}
Let $\Top$ denote the category of compactly generated weak Hausdorff topological spaces. For us a \emph{topological category} will mean a (not necessarily small) category enriched in $\Top$. This means that the morphism sets are topologized so as to be objects in $\Top$ and that composition is continuous.  For a small topological category $\cK$, a \emph{$\cK$-space} is by definition a continuous functor $X\colon\cK\to \Top$, and we write $\Top^{\cK}$ for the topological category of $\cK$-spaces. 

In the following we briefly review some basic constructions on $\cK$-spaces, referring the reader to \cite{Hollender-V_modules} for more details. Given a $\cK$-space $X$ and a $\cK^{\op}$-space $Y$, the bar construction $B(Y,\cK,X)$ is defined as the realization of the simplicial space with $p$-simplices
\[
B_p(Y,\cK,X)=\coprod_{\bld k_0,\dots,\bld k_p} Y(\bld k_0)\times \cK(\bld k_1,\bld k_0)\times\dots\times \cK(\bld k_p,\bld k_{p-1})\times X(\bld k_p)
\]  
where the coproduct is over all ($p+1$)-tuples of objects in $\cK$. The simplicial structure maps are of the usual bar construction type as detailed in \cite[Section~3]{Hollender-V_modules}. 

Let $F\colon \cK\to \cL$ be a continuous functor between small topological categories. Given a $\cK$-space $X$, the homotopy left Kan extension along $F$ is by definition the $\cL$-space $F_*^h(X)$ defined by 
\[
F_*^h(X)(\bld l)=B(\cL(F(-),\bld l),\cK,X)
\] 
for $\bld l$ an object in $\cL$ and $\cL(F(-),\bld l)$ the associated $\cK^{\op}$-space.
As a special case of this construction, the 
\emph{bar resolution} $\overline X$ of a $\cK$-space $X$ is the homotopy left Kan extension along the identity functor on $\cK$,
\[
\overline X(\bld k)=B(\cK(-,\bld k),\cK,X).
\]
There is canonical ``evaluation'' map of $\cK$-spaces $\epsilon\colon\overline X\to X$ which is a level-wise equivalence, cf.\ \cite[Proposition~3.1]{Hollender-V_modules}. Another special case is the 
homotopy colimit of a $\cK$-space $X$ defined by
\[
\hocolim_{\cK}X=B(*,\cK,X)
\] 
where $*$ denotes the terminal $\cK$-space. In particular, the homotopy colimit of $*$ is the classifying space $B\cK=B(*,\cK,*)$. It follows from the definitions that there is a natural homeomorphism 
$\colim_{\cK}\overline X\cong \hocolim_{\cK}X$. For convenience we often write $X_{h\cK}$ instead of  
$\hocolim_{\cK}X$.

\subsection{The $\boxtimes$-product on $\Top^{\cK}$}\label{subset:boxtimes-product}
We say that a topological category $\cK$ is \emph{permutative} if it
has a continuous symmetric strict monoidal product $\oplus$ with strict unit
$0$.  A permutative structure on $\cK$ induces a symmetric
monoidal convolution product on $\Top^{\cK}$: Given $\cK$-spaces $X$
and $Y$, we define $X\boxtimes Y$ to be the left Kan extension of the
product diagram
\[
X\times Y\colon \cK\times\cK \xr{X\times Y} \Top\times\Top \xr{\times}\Top
\] 
along $\oplus\colon \cK\times\cK\to \cK$. Thus, with the $\otimes$-notation from 
\cite{Hollender-V_modules}, we have
\[
(X\boxtimes Y)(\bld k)=\cK(-\oplus-,\bld k)\otimes_{\cK\times\cK}(X\times Y),
\]
also known as the coend of the evident $(\cK\times\cK)^{\op}\times (\cK\times\cK)$-diagram. By the universal property of the left Kan extension, a map of $\cK$-spaces $X\boxtimes Y\to Z$ amounts to a natural transformation of $\cK\times\cK$-diagrams
\[
X(\bld h)\times Y(\bld k)\to Z(\bld h\oplus \bld k).
\]
We write $U^{\cK}=\cK(0,-)$ for the $\cK$-space defined by the monoidal unit for $\cK$.
The following proposition is proved by arguments that are by now quite standard.


\begin{proposition}
The $\boxtimes$-product makes $\Top^{\cK}$ a closed symmetric monoidal category with 
$U^{\cK}$ as monoidal unit. \qed 
\end{proposition}

We use the term \emph{$\cK$-space monoid} for a monoid in $\Top^{\cK}$ with respect to the $\boxtimes$-product.
The next lemma is a formal consequence of the universal properties of the free $\cK$-spaces $F_{\bld k}^{\cK}(K)=\cK(\bld k,-)\times K$ and the $\boxtimes$-product.
\begin{lemma}\label{lem:free-prod}
Given a pair of spaces $K$ and $L$, and a pair of objects $\bld h$ and $\bld k$ in $\cK$, there is a natural isomorphism
\[
F_{\bld h}^{\cK}(K)\boxtimes F_{\bld k}^{\cK}(L)\cong F_{\bld h\oplus \bld k}^{\cK}
(K\times L).
\eqno\qed
\]
\end{lemma}

\subsection{Categories of linear isometries}\label{sec:isometries}
Let $\cV$ be the topological category with objects the standard inner product spaces $\mathbb R^n$ for $n\geq 0$ and morphisms the (not necessarily surjective) isometries. The canonical identification of $\mR^m\oplus \mR^n$ with $\mR^{m+n}$ makes $\cV$ a permutative topological category with symmetry isomorphisms $\chi_{m,n}\colon\mathbb R^m\oplus \mathbb R^n\to \mathbb R^n\oplus \mathbb R^m$ defined by $\chi_{m,n}(u,v)=(v,u)$ for $u\in \mathbb R^m$ and $v\in \mathbb R^n$. It is convenient to identify the objects of $\cV$ with the natural numbers $n\geq 0$. Thus, we shall write 
$\cV(m,n)$ instead of $\cV(\mR^m,\mR^n)$ and $m\oplus n$ instead of $\mR^m\oplus \mR^n$.
With this notation, the orthogonal group $O(n)$ is the group of endomorphisms $\cV(n,n)$ of~$n$. We write $\cO$ for the subcategory of isometric isomorphisms in $\cV$.

\begin{definition}
Let $\cW=\cO^{-1}\cO$ be Quillen's localization construction \cite{Grayson-higher} applied to $\cO$. 
In detail, we specify that the objects of $\cW$ be pairs of natural numbers $(n_1,n_2)$. The morphism space $\cW((m_1,m_2), (n_1,n_2))$ is non-empty if and only if there exists a natural number $m$ such that $m_1+m=n_1$ and $m_2+m=n_2$. In this case
\begin{equation}\label{eq:endomorphisms-of-W}
\cW((m_1,m_2), (n_1,n_2)) 
=\big(\cO(m_1\oplus m,n_1)\times \cO(m_2\oplus m,n_2)\big)/O(m)
\end{equation}
where the orthogonal group $O(m)$ acts diagonally from the right via the inclusions 
$O(m)\to O(m_i\oplus m)$ for $i=1,2$, extending an isometry of $\mR^{m}$ by the identity on 
$\mR^{m_i}$. The elements in this morphism space are written in the form $[m,\sigma_1,\sigma_2]$ for 
$\sigma_1$ in $O(n_1)$ and $\sigma_2$ in $O(n_2)$. With this notation composition is given by
\[
[n,\tau_1,\tau_2]\circ [m,\sigma_1,\sigma_2]=
[m\oplus n,\tau_1\circ(\sigma_1\oplus\id_{\mR^{n}}), \tau_2\circ(\sigma_2\oplus\id_{\mR^{n}})].
\]
\end{definition}

The next lemma gives a geometric interpretation of the morphisms spaces in $\cW$. 
Here we write $\mR^n\ominus V$ for the orthogonal complement of a linear subspace $V$ in $\mR^n$. 
\begin{lemma}
Suppose that $m_1+m=n_1$ and $m_2+m=n_2$. Then the map
\[
\cW((m_1,m_2), (n_1,n_2))\to \cV(m_1,n_1)\times 
\cV(m_2,n_2),\quad [m,\sigma_1,\sigma_1]\mapsto (\sigma_1|\mR^{m_1},\sigma_2|\mR^{m_2})
\]
is a fiber bundle whose fiber over an element $(\alpha_1,\alpha_2)$  can be identified with the space of isometric isomorphisms between $\mR^{n_1}\ominus \alpha_1(\mR^{m_1})$ and 
$\mR^{n_2}\ominus\alpha_2(\mR^{m_2})$.\qed
\end{lemma}

The permutative structure of $\cW$ is defined on objects by 
\[
(m_1,m_2)\oplus (n_1,n_2)=(m_1\oplus n_1,m_2\oplus n_2).
\]
For a pair of morphisms 
\[
[m,\sigma_1,\sigma_2]\colon (m_1,m_2)\to (m'_1,m'_2),\quad 
[n,\tau_1,\tau_2]\colon (n_1,n_2)\to (n'_1,n'_2),
\]
we define $[m,\sigma_1,\sigma_2]\oplus[n,\tau_1,\tau_2]$ to be the morphism
\[
[m\oplus n,(\sigma_1\oplus\tau_1)\circ (\id_{\mR^{m_1}}\oplus \chi_{n_1,m}\oplus \id_{\mathbb R^{n}}),
(\sigma_2\oplus\tau_2)\circ (\id_{\mR^{m_2}}\oplus \chi_{n_2,m}\oplus \id_{\mathbb R^{n}})].
\]
The object $(0,0)$ is a strict unit for this product and we have the coordinate-wise symmetry isomorphism on 
$\cW$ defined by
\[
[\chi_{m_1,n_1},\chi_{m_2,n_2}]\colon (m_1,m_2)\oplus(n_1,n_2)\to(n_1,n_2)\oplus(m_1,m_2).
\]
We define the degree of an object $(n_1,n_2)$ to be the integer $n_2-n_1$ and 
write $\cW_{\{d\}}$ for the full subcategory of $\cW$ with objects of degree $d$. Thus, $\cW$ decomposes as a coproduct of the categories $\cW_{\{d\}}$ for $d\in \mathbb Z$ and the permutative structure restricts to functors $\oplus\colon\cW_{\{d\}}\times \cW_{\{e\}}\to \cW_{\{d+e\}}$.


\subsection{The Grothendieck construction}\label{subsec:Grothendieck-construction}
It will be convenient to have a description of the $d$th component $\cW_{\{d\}}$ in terms of a certain Grothendieck construction (in the sense of Thomason \cite{Thomason-homotopy-colimt}). Let us write 
$\TopCat$ for the (topological) category of small topological categories. Consider the functor $O\colon \cV\to \TopCat$ that takes $\mR^n$ to $O(n)$, thought of as a topological category with a single object, and that takes an isometry $\alpha\colon \mR^m\to \mR^n$ to the functor (that is, group homomorphism) $O_{\alpha}\colon O(m)\to O(n)$ defined as follows: an element 
$a\in O(m)$ is mapped to the element $O_{\alpha}(a)\in O(n)$ determined by the commutativity of the diagram
\[
\xymatrix@-.5pc{
\mR^n \ar[r]^{O_{\alpha}(a)} & \mR^n\\
\mR^m\oplus (\mR^n\ominus\alpha(\mR^m)) \ar[r]^{a\oplus\id}\ar[u]^{\{\alpha,\mathrm{inc}\}}& \mR^m
\oplus(\mR^n\ominus \alpha(\mR^m))\ar[u]_{\{\alpha,\mathrm{inc}\}}.
}
\]
Here $\mR^n\ominus\alpha(\mR^m)$ again denotes the orthogonal complement of $\alpha(\mR^m)$ and 
$\mathrm{inc}$ is the inclusion. 
For $d\geq 0$ we write $O(d\oplus-)\colon \cV\to \TopCat$ for the functor obtained from $O$ by precomposition with the endofunctor $d\oplus -$ on $\cV$.
The functor $O(d\oplus -)$ gives rise to the topological category $\VintOd$, the Grothendieck construction, with objects the natural numbers $n\geq 0$ and morphisms 
$(a,\alpha)\colon m\to n$ given by a linear isometry $\alpha\colon \mR^m\to \mR^n$ in $\cV$ and an element $a\in O(d\oplus n)$. Composition of morphisms is defined by
\[
(b,\beta)\circ(a,\alpha)=(b\cdot O_{d\oplus\beta}(a),\beta\alpha).
\] 
For $d=0$ we use the notation $\VintO$ instead of $\VintOd$. The direct sum on objects and morphisms inherited from $\cV$ makes $\VintO$ a permutative category with symmetry isomorphism $(\id_{\mR^{n+m}},\chi_{m,n})\colon m\oplus n\to
n\oplus m$ induced by the symmetry isomorphism $\chi_{m,n}$ from $\cV$. 

\begin{lemma}\label{lem:VintO-W0-equivalence}
The categories  $\VintO$ and $\cW_{\{0\}}$ are isomorphic as permutative topological categories, and for $d\geq 1$ there are canonical isomorphisms of topological categories $\VintOd\cong \cW_{\{d\}}\cong\cW_{\{-d\}}$.
\end{lemma}
\begin{proof}
The isomorphism $\VintOd\to \cW_{\{d\}}$ takes an object $n$ to $(n,d\oplus n)$. Using the canonical identification of $\cV(m,n)$ with $\cO(m\oplus l,n)/O(l)$, for $m\oplus l=n$, the effect on morphisms is defined by the homeomorphism
\[
O(d\oplus n)\times \cO(m\oplus l,n)/O(l)\to \big(\cO(m\oplus l,n)\times\cO(d\oplus m\oplus l,d \oplus n)\big)/O(l)
\]
taking $(a,[l,\alpha])$ to $[l,\alpha, a\circ(d\oplus \alpha)]$. The inverse homeomorphism takes $[l,\alpha_1,\alpha_2]$ to $(\alpha_2\circ(d\oplus \alpha_1^{-1}), [l,\alpha_1])$. 
Finally, the obvious ``inversion'' automorphism of $\cW$, taking $(n_1,n_2)$ to $(n_2,n_1)$, restricts to an isomorphism $\cW_{\{d\}}\cong\cW_{\{-d\}}$.
\end{proof}
Introducing the Grothendieck construction $\VintOd$ in this context allows us to invoke Thomason's homotopy colimit theorem in the next lemma. Let $\cN$ be the ordered set of natural numbers and define a functor $\cN\to \cV$ by mapping $n$ to $\mR^n$ and $n\to n+1$ to the inclusion $\mR^n\oplus\mR^0\to \mR^n\oplus\mR$.

\begin{proposition}\label{prop:Grothendieck-hocolim-equivalence}
For a $\cW$-space $X$ and $d\geq 0$, there are natural weak homotopy equivalences
\[
\begin{aligned}
&\hocolim_{n\in\cN}X(n,d\oplus n)_{hO(d\oplus n)}\xr{\sim}\hocolim_{n\in \cV}X(n,d\oplus n)_{hO(d\oplus n)}
\xr{\sim}X_{h\cW_{\{d\}}},\\
&\hocolim_{n\in\cN}X(d\oplus n,n)_{hO(d\oplus n)}\xr{\sim}\hocolim_{n\in \cV}X(d\oplus n,n)_{hO(d\oplus n)}
\xr{\sim}X_{h\cW_{\{-d\}}},
\end{aligned}
\]
where $X_{h\cW_{\{-\}}}$ is the homotopy colimit of $X$ restricted to the stated subcategory, and $(-)_{hO(d\oplus n)}$ denotes the homotopy orbits of the left $O(d\oplus n)$-action in the first or second variable.
\end{proposition}
\begin{proof}
The first maps in each composition are induced by the functor $\cN\to \cV$ and are weak homotopy equivalences by  \cite[Lemma~7.3]{Lind-diagram}. The second maps are defined in analogy with the map considered by Thomason 
\cite[Lemma~1.2.1]{Thomason-homotopy-colimt}, using the isomorphisms from Lemma~\ref{lem:VintO-W0-equivalence}. They are weak equivalences by a slight generalization of the argument used in the proof of \cite[Theorem~1.2]
{Thomason-homotopy-colimt} (see also \cite[Proposition~6.2]{Hollender-V_modules} and the dual argument in \cite[Theorem~2.3]{Schlichtkrull-cyclotomic}).
\end{proof}

Applied to the terminal $\cW$-space, the above proposition determines the homotopy type of 
$B\cW$, 
\[
\hocolim_{n\in \cN}BO(n)\xr{\sim} \hocolim_{n\in \cN}BO(d\oplus n)\xr{\sim}B\cW_{\{d\}}\cong B\cW_{\{-d\}}.
\]
It follows that $B\cW$ is a model of $BO\times \mZ$, the classifying space for $KO$-theory, as was first observed by Thomason
based on the analysis in \cite{Grayson-higher}. 

\section{The homotopy theory of $\cW$-spaces}\label{sec:W-homotopy-theory}
In this section we set up the basic homotopy theory of $\cW$-spaces. This is analogous to the homotopy theory of $\cJ$-spaces considered in \cite{Sagave-S_diagram}. 
Much of this material is quite standard and we only cover the minimum needed for the rest of the paper.

\subsection{The $\cW$-model structure}
We first consider the level model structure on $\Top^{\cW}$ and say that a map of $\cW$-spaces $X\to Y$ is a level equivalence (respectively a level fibration) if $X(n_1,n_2)\to Y(n_1,n_2)$ is a weak homotopy equivalence (respectively a fibration) for all objects $(n_1,n_2)$ in $\cW$. We say that $X\to Y$ is a cofibration if it has the left lifting property with respect to maps of $\cW$-spaces that are both level equivalences and level fibrations. Consider for each object $(d_1,d_2)$ the pair of adjoint functors
\[
F_{(d_1,d_2)}^{\cW}\colon \Top \rightleftarrows \Top^{\cW}:\! \Ev_{(d_1,d_2)}
\]
where $ \Ev_{(d_1,d_2)}(X)= X(d_1,d_2)$ and  $F_{(d_1,d_2)}^{\cW}(K)=\cW((d_1,d_2),-)\times K.$
Let $I$ be the standard set of generating cofibrations for $\Top$ of the form $S^{n-1}\to D^n$ for 
$n\geq 0$, and let $J$ be the standard set of generating acyclic cofibrations of the form 
$D^n\to D^n\times I$ for $n\geq 0$.
We let $FI$ (respectively $FJ$) denote the set of maps in $\Top^{\cW}$ of the form $F_{(d_1,d_2)}^{\cW}(i)$ for $i$ an element in $I$ (respectively $J$). The following result is standard (see for instance the analogous results for discrete index categories in \cite[Theorem~11.6.1]{Hirschhorn-model} and based topological index categories in \cite[Theorem~6.5]{MMSS}).    

\begin{proposition}
The level equivalences, level fibrations, and cofibrations specify a cofibrantly generated model structure on $\Top^{\cW}$ with generating cofibrations $FI$ and generating acyclic cofibrations $FJ$. \qed
\end{proposition}

We shall be more interested in a certain localization of the level model structure and say that a map of 
$\cW$-spaces $X\to Y$ is a
\begin{itemize}
\item
$\cW$-equivalence if the induced map of homotopy colimits $X_{h\cW}\to Y_{h\cW}$ is a weak homotopy equivalence,
\item
$\cW$-fibration if it is a level fibration and the diagram 
\begin{equation}\label{eq:hty-cart-for-W-fib}
\xymatrix@-1pc{
X(m_1,m_2)\ar[r] \ar[d]&X(n_1,n_2)\ar[d]\\
Y(m_1,m_2) \ar[r] &Y(n_1,n_2)
}
\end{equation}
is homotopy cartesian for any morphism $(m_1,m_2)\to (n_1,n_2)$ in $\cW$,
\item
$\cW$-cofibration if it is a cofibration in the level model structure. 
\end{itemize}    

\begin{proposition}\label{prop:W-model-str}
The $\cW$-equivalences, $\cW$-fibrations, and $\cW$-cofibrations specify a cofibrantly generated proper model structure on $\Top^{\cW}$.
\end{proposition}
We shall refer to this as the \emph{$\cW$-model structure} on $\Top^{\cW}$. The $\cW$-cofibrations will usually be referred to simply as cofibrations. 
The proof of the proposition is based on the following lemma. It is a variant of~\cite[Proposition 4.4]{Rezk_SS-simplicial} that applies to continuous functors $\cW \to \Top$ as opposed to functors defined only on discrete indexing categories.
\begin{lemma}\label{lem:fib-of-W-spaces-and-hty-car-sq}
Let $X \to Y$ be a map in $\Top^{\cW}$ such that the left hand square in
\[\xymatrix@-1pc{X(m_1,m_2) \ar[r] \ar[d]_{[m,\sigma_1,\sigma_2]_*} & Y(m_1,m_2)\ar[d]^{[m,\sigma_1,\sigma_2]_*} & & & X(m_1,m_2) \ar[r] \ar[d] & Y(m_1,m_2)\ar[d] \\ X(n_1,n_2) \ar[r] & Y(n_1,n_2) & & & X_{h\cW} \ar[r] & Y_{h\cW} }\] is homotopy cartesian for every morphism $[m,\sigma_1,\sigma_2]\colon (m_1,m_2) \to (n_1,n_2)$ in $\cW$. Then the right hand square is homotopy cartesian for every object $(m_1,m_2)$ in $\cW$.
\end{lemma}
\begin{proof}
Suppose first that $(m_1,m_2)$ has the form $(m,d\oplus m)$ for $d\geq 0$. By Proposition~\ref{prop:Grothendieck-hocolim-equivalence} it suffices to show that the outer square in the commutative diagram 
\[
\xymatrix@-1pc{ 
X(m,d\oplus m) \ar[r] \ar[d] & X(m,d\oplus m)_{hO(d\oplus m)} \ar[r]\ar[d] &\displaystyle \hocolim_{n\in \cN}X(n,d\oplus n)_{hO(d\oplus n)}\ar[d] \\ 
Y(m,d\oplus m) \ar[r] & Y(m,d\oplus m)_{hO(d\oplus m)} \ar[r] & \displaystyle\hocolim_{n\in \cN}Y(n,d\oplus n)_{hO(d\oplus n)} \\ 
}
\] 
is homotopy cartesian for every $m\geq 0$. Here the left hand square is homotopy cartesian by~\cite[Theorem 7.6]{May-classifying}, whereas the right hand square is homotopy cartesian by the version of the lemma for discrete indexing categories, see~\cite[Proposition~4.4]{Rezk_SS-simplicial} and~\cite[Lemma~6.12 and Remark~6.13]{Sagave-S_diagram}. The proof when $(m_1,m_2)$ has the form $(d\oplus m,m)$ is completely analogous.
\end{proof}

\begin{proof}[Proof of Proposition~\ref{prop:W-model-str}]
  This is essentially~\cite[Proposition~6.16]{Sagave-S_diagram} with $\cK = \cW$ and $\cA$ being the subcategory of all identity morphisms in $\cW$. The only change, forced by the topological enrichment of $\cW$, is that Lemma~\ref{lem:fib-of-W-spaces-and-hty-car-sq} replaces~\cite[Lemma 6.12]{Sagave-S_diagram}. Properness follows as in~\cite[Section 11]{Sagave-S_diagram}, where in the proof of~\cite[Proposition 11.3]{Sagave-S_diagram} we again replace~\cite[Lemma~6.12]{Sagave-S_diagram} by Lemma~\ref{lem:fib-of-W-spaces-and-hty-car-sq}. 
\end{proof}
The following lemma justifies thinking of the homotopy colimit functor as a derived version of the colimit functor. 

\begin{lemma}\label{lem:hocolim-colim}
  For a cofibrant $\cW$-space $X$, the canonical map $X_{h\cW}\to \colim_{\cW}X$ is a weak equivalence.
\end{lemma}
\begin{proof}
The map in the lemma is obtained by applying 
$\colim_{\cW}\colon \Top^{\cW}\to \Top$ to the evaluation map 
$\epsilon \colon \overline X\to X$. Since $\epsilon$ is a level equivalence and 
  $\colim_{\cW}$ is a left Quillen functor, the claim follows from the fact that 
  $\overline X$ is cofibrant by Lemma~\ref{lem:bar-resolution-cofibrant} below. 
  \end{proof}

The next lemma provides an important source of cofibrant $\cW$-spaces.
\begin{lemma}\label{lem:homotopy-Kan-cofibrant}
Let $\cK$ be a small topological category and let $F\colon \cK\to\cW$ be a continuous functor. Assume that the morphism spaces in $\cK$ are cofibrant and that the inclusions of the identity morphisms are cofibrations. Let $X$ be a $\cK$-space such that $X(\bld k)$ is cofibrant for all objects $\bld k$ in $\cK$. Then the homotopy left Kan extension $F_*^h(X)$ is a cofibrant $\cW$-space.  
\end{lemma}
\begin{proof}
Consider in general a simplicial $\cW$-space $Z_{\bullet}$ and observe that the geometric realization 
$Z=|Z_{\bullet}|$ admits a filtration by $\cW$-spaces $Z^{(n)}$ for $n\geq 0$, such that $Z^{(0)}=Z_0$, 
$\colim_n Z^{(n)}=Z$, and there are pushout diagrams 
\[
\xymatrix@-.5pc{
SZ_{n-1}\times \Delta^n\cup_{SZ_{n-1}\times\partial \Delta^n}Z_n\times \partial\Delta^n \ar[r]\ar[d] & Z_n\times \Delta^n\ar[d]\\
Z^{(n-1)}\ar[r] & Z^{(n)}
}
\]  
where $SZ_{n-1}\subseteq Z_n$ is the union of the degenerate sub $\cW$-spaces $s_i(Z_{n-1})$ for $0\leq i\leq n$. Thus, we see that the $\cW$-space $Z$ is cofibrant provided that $Z_0$ is cofibrant and the maps $SZ_{n-1}\to Z_n$ are cofibrations. In the case at hand, $F_*^h(X)$ is the geometric realization of a simplicial $\cW$-space with $p$-simplices 
\[
\coprod_{\bld k_0,\dots,\bld k_p}
F^{\cW}_{F(\bld k_0)}\big(\cK(\bld k_1,\bld k_0)\times\dots\times \cK(\bld k_p,\bld k_{p-1})\times X(\bld k_p)\big),
\] 
the coproduct being over all $(p+1)$-tuples $\bld k_0$,\dots, $\bld k_p$ of objects in $\cK$. Since the functors $F_{F(\bld k_0)}^{\cW}$ are left Quillen functors, the assumptions in the lemma ensure that this is a cofibrant $\cW$-space for all $p$ and in particular for $p=0$. Furthermore, the inclusions of the degenerate sub $\cW$-spaces are obtained by applying the functors $F^{\cW}_{F(\bld k_0)}$ to the evident cofibrations in $\Top$, hence are  cofibrations of $\cW$-spaces.
\end{proof}

This lemma applies in particular to the identity functor on $\cW$:
\begin{lemma}\label{lem:bar-resolution-cofibrant}
Let $X$ be a $\cW$-space and suppose that $X(n_1,n_2)$ is cofibrant as an object in $\Top$ for all 
$(n_1,n_2)$. Then the bar resolution $\overline X$ is a cofibrant $\cW$-space.\qed
\end{lemma}

\subsection{The $\boxtimes$-product on $\Top^{\cW}$} 
Specializing the general theory in Section~\ref{subset:boxtimes-product} to the permutative structure of $\cW$, we get a symmetric monoidal product $\boxtimes$ on $\Top^{\cW}$ with $U^{\cW}$ as monoidal unit. 

\begin{proposition}\label{prop:boxtimes-w-cof-pres-w-equiv}
If $X$ is a cofibrant $\cW$-space, then the functor $X\boxtimes(-)$ preserves $\cW$-equivalences. 
\end{proposition}
\begin{proof}
Most parts of this argument are analogous to the corresponding statements for
discrete indexing categories in~\cite[Proposition 8.5]{Sagave-S_diagram}. In particular, a cell reduction argument reduces the claim to the case where $X = F^{\cW}_{(k_1,k_2)}(L)$ for a CW complex $L$. If $Y$ is a $\cW$-space, an inspection of the coequalizer defining $\boxtimes$ shows that $(F^{\cW}_{(k_1,k_2)}(L) \boxtimes Y)(n_1,n_2)$ can be identified with the quotient space
\begin{equation}\label{eq:FWk-boxtimes-as-quotient}
 \cW((k_1,k_2) \oplus (l_1,l_2), (n_1,n_2) )\times_{\cW((l_1,l_2) ,(l_1,l_2) )}(L\times Y(l_1,l_2))
\end{equation}
if there exists an object $(l_1,l_2)$ with $(k_1,k_2)\oplus (l_1,l_2) = (n_1,n_2)$. If no such object exists, 
$(F^{\cW}_{(k_1,k_2)}(L) \boxtimes Y)(n_1,n_2) $ is empty. Since 
$\cW((k_1,k_2) \oplus (l_1,l_2), (n_1,n_2) ) $ is a free $\cW((l_1,l_2) ,(l_1,l_2) )$-CW complex, it follows that $F^{\cW}_{(k_1,k_2)}(L) \boxtimes(-)$ preserves level equivalences of $\cW$-spaces. 
Hence the bar resolution $\ovl{Y} \to Y$ induces a level equivalence $F^{\cW}_{(k_1,k_2)}(L) \boxtimes \ovl{Y}\to F^{\cW}_{(k_1,k_2)}(L) \boxtimes Y$. Writing $F = (k_1,k_2)\oplus(-) \colon \cW \to \cW$, we can identify $F^{\cW}_{(k_1,k_2)}(*)\boxtimes(-)$ with the left Kan extension along $F$. Moreover, the isomorphism in ~\cite[3.1 Proposition(i)]{Hollender-V_modules} shows that the $\cW$-space $F^{\cW}_{(k_1,k_2)}(L) \boxtimes \ovl{Y}$ is isomorphic to $L\times F^h_*(Y)$, the product of $L$ with the homotopy left Kan extension of $Y$ along $F$. Since there is also a natural weak equivalence 
\[
\hocolim_{\cW} (L\times  F^h_*(Y)) \xrightarrow{\sim} L \times \hocolim_{\cW}Y
\]  
it follows that $F^{\cW}_{(k_1,k_2)}(L)\boxtimes(-)$ preserves $\cW$-equivalences as claimed. 
\end{proof}

\begin{corollary}\label{cor:pproduct-and-monoid-axiom}
The $\cW$-model structure satisfies the pushout-product axiom and the monoid axiom. 
\end{corollary}
\begin{proof}
This is analogous to~\cite[Propositions 8.4 and 8.6]{Sagave-S_diagram}, with
the previous proposition replacing~\cite[Proposition 8.2]{Sagave-S_diagram}. 
\end{proof}

We shall view the homotopy colimit functor $(-)_{h\cW}\colon \Top^{\cW}\to \Top$ as a lax monoidal functor with monoidal product defined by the natural map
\[
X_{h\cW}\times Y_{h\cW}\xr{\cong} \big(X\times Y\big)_{h(\cW\times \cW)} \to 
\big(X\boxtimes Y\circ\oplus\big)_{h(\cW\times \cW)} \to \big(X\boxtimes Y\big)_{h\cW}.
\]
Arguing as in the case of $\cI$-spaces \cite[Lemma~2.25]{Sagave-S_group-completion} we get the next result.
\begin{lemma}\label{lem:hW-monodial-equivalence}
The natural map $X_{h\cW}\times Y_{h\cW}\to\big(X\boxtimes Y\big)_{h\cW}$ is a weak homotopy equivalence provided that either $X$ or $Y$ is cofibrant. \qed
\end{lemma}

We say that a map of $\cW$-spaces $X \to Y$ is a \emph{positive $\cW$-fibration} if the map $X(n_1,n_2)\to Y(n_1,n_2)$ is a fibration for every object $(n_1,n_2)$ with $n_1\geq 1$, and if 
the square~\eqref{eq:hty-cart-for-W-fib} is homotopy cartesian for every morphism $(m_1,m_2)\to (n_1,n_2)$ with $m_1\geq 1$. Using Proposition~\ref{prop:W-model-str}, one can show that the positive $\cW$-fibrations and the $\cW$-equivalences participate in a cofibrantly generated proper model structure which we shall refer to as the \emph{positive $\cW$-model structure}. The cofibrations in this model structure are called \emph{positive cofibrations}. 

Since every positive cofibration is a cofibration in the (absolute)
$\cW$-model structure of Proposition~\ref{prop:W-model-str}, the
statements of Proposition~\ref{prop:boxtimes-w-cof-pres-w-equiv} and
Corollary~\ref{cor:pproduct-and-monoid-axiom} also hold for the positive $\cW$-model structure. The relevance of the
positive $\cW$-model structure comes from the following result.
\begin{theorem}\label{thm:positive-model-structure}
The category of commutative $\cW$-space monoids admits a positive $\cW$-model structure where a map is a fibration or a weak equivalence if and only if the underlying map of $\cW$-spaces is a positive fibration or $\cW$-equivalence.
\end{theorem}
\begin{proof}
Using the homotopy invariance of the $\boxtimes$-product from Proposition ~\ref{prop:boxtimes-w-cof-pres-w-equiv}, the construction of this model structure is completely analogous to its discrete counterpart in~\cite[Lemma 9.5]{Sagave-S_diagram}. The only difference is that we here use the expression~\eqref{eq:FWk-boxtimes-as-quotient} considered in the proof of Proposition~\ref{prop:boxtimes-w-cof-pres-w-equiv} to see that the $\Sigma_i$-action on $F^{\cW}_{(k_1,k_2)^{\concat i}}\boxtimes *$ is free if $k_1 \geq 1$. 
\end{proof}

\section{$\cW$-spaces and graded orthogonal spectra}\label{sec:W-spaces-SpO}
In this section we set up the adjunction relating $\cW$-spaces to orthogonal spectra. This is the orthogonal analogue of the adjunction in \cite{Sagave-S_diagram} relating $\cJ$-spaces to symmetric spectra. We write $S^n$ for the one-point compactification of $\mR^n$ throughout. 

\subsection{Orthogonal spectra}
Following \cite{MMSS}, an orthogonal spectrum $E$ is a sequence of based spaces $E_n$ for $n\geq 0$, together with a based left $O(n)$-action on~$E_n$ and a family of based structure maps 
$E_n\wedge S^1\to E_{n+1}$ for $n\geq 0$ such that the iterated structure maps 
$E_m\wedge S^n\to E_{m+n}$ are  $(O(m)\times O(n))$-equivariant. Here $O(m)\times O(n)$ acts on 
$E_{m+n}$ via the inclusion $O(m)\times O(n)\to O(m+n)$. A map of orthogonal spectra $E\to E'$ is a family of $O(n)$-equivariant based maps $E_n\to E'_n$ that are compatible with the structure maps. We use the notation $\SpO$ for the (topological) category of orthogonal spectra and equip this with the stable model structure established in \cite{MMSS}.
The sphere spectrum $\mS$ is the orthogonal spectrum with $n$th space $S^n$ and the obvious structure maps. It is proved in \cite{MMSS} that there is a smash product $\wedge$ of orthogonal spectra making $\SpO$ a closed symmetric monoidal category with $\mS$ as its monoidal unit.

The category $\SpO$ is related to the category of based spaces $\Top_*$ by a pair of adjoint functors 
$F_d\colon \Top_*\rightleftarrows \SpO:\! \Ev_d$ for each natural number $d$. The evaluation functor $\Ev_d$ is defined
by $\Ev_d(E) = E_d$, and its left adjoint $F_d$ takes a based space $K$ to the orthogonal spectrum with $n$th term
\[
F_d(K)_n=\cO(d\oplus(n-d),n)_+\wedge_{O(n-d)}K\wedge S^{n-d}
\] 
if $n\geq d$ and $F_d(K)_n=*$ otherwise (where $(-)_+$ indicates the addition of a disjoint base point). Here we think of $\mR^{n-d}$ as the orthogonal complement of 
$\mR^d$ included as the first $d$ coordinates in $\mR^n$.
The structure maps are defined by  
\[
F_d(K)_n\wedge S^1\to F_d(K)_{n+1},\quad [a,x,u,t]\mapsto[a\oplus \id_{\mR},x,(u,t)]
\]
for $a\in \cO(d\oplus(n-d),n)$, 
$x\in K$, $u\in S^{n-d}$, and $t\in S^1$. 
For $d=0$ this construction gives the orthogonal suspension spectrum $F_0(K)$ with $n$th space $K\wedge S^n$. 
It follows from \cite[Lemma 1.8]{MMSS} that the canonical family of maps
\[
\begin{gathered}
F_d(K)_m\wedge F_e(L)_n\to F_{d+e}(K\wedge L)_{m+n}, \\ ([a,x,u],[b,y,v])\mapsto 
[\left(a\oplus b\right) \left(\id_{\mR^d}\oplus\chi_{e,m-d}\oplus\id_{\mR^{n-e}}\right), (x,y),(u,v)] 
\end{gathered}
\] 
induces an isomorphism of orthogonal spectra
\begin{equation}\label{eq:F-monoidal-iso}
F_d(K)\wedge F_e(L)\cong F_{d+e}(K\wedge L).
\end{equation}

\subsection{The adjunction to $\cW$-spaces} In order to set up the
adjunction relating $\cW$-spaces to orthogonal spectra,
we first define a functor $\cW^{\op}\to \SpO$. On objects this is defined by mapping $(n_1,n_2)$ to $F_{n_1}(S^{n_2})$. Specifying the effect on morphism spaces  
\[
\cW((m_1,m_2),(n_1,n_2))\to \SpO(F_{n_1}(S^{n_2}),F_{m_1}(S^{m_2}))
\]
amounts by adjointness to specifying a family of based maps
\begin{equation}\label{eq:adjoint-Wop-SpO}
\cW((m_1,m_2),(n_1,n_2))_+\wedge S^{n_2}\to F_{m_1}(S^{m_2})_{n_1}.
\end{equation}
Given a natural number $m$ such that $m_1+m=n_1$ and $m_2+m=n_2$, there is an $O(m)$-equivariant 
map
\[
\big(\cO(m_1\oplus m,n_1)\times \cO(m_2\oplus m,n_2)\big)_+\wedge S^{n_2}
\to \cO(m_1\oplus m,n_1)_+\wedge S^{m_2}\wedge S^m
\]
defined by $(\sigma_1,\sigma_2,w)\mapsto (\sigma_1,\sigma_2^{-1}(w))$. We define the map~\eqref{eq:adjoint-Wop-SpO} to be the induced map of $O(m)$-orbit spaces. The following lemma is the orthogonal analogue of the corresponding result for symmetric spectra proved in \cite[Section~4.21]{Sagave-S_diagram}.
\begin{lemma}\label{lem:FWop-functor}
The structure maps defined above make
\[
F_{-}(S^{-})\colon \cW^{\op}\to \SpO,\quad (n_1,n_2)\mapsto F_{n_1}(S^{n_2})
\]
a strong symmetric monoidal functor. 
\end{lemma}
\begin{proof}
Using the explicit description of the orthogonal spectra $F_{n_1}(S^{n_2})$, one checks that this construction indeed defines a continuous functor on $\cW^{\op}$. The symmetric monoidal structure is given by the canonical isomorphism $\mS\to F_0(S^0)$ and the natural isomorphisms
\[
F_{m_1}(S^{m_2})\wedge F_{n_1}(S^{n_2})\to F_{m_1+n_1}(S^{m_2+n_2})
\]
defined as in \eqref{eq:F-monoidal-iso}. 
\end{proof}

Now we apply a general principle for defining left adjoint functors out of diagram categories. For a $\cW$-space $X$, let $\mS^{\cW}[X]$ be the orthogonal spectrum defined as the coend 
\[
\mS^{\cW}[X]=\int^{(n_1,n_2)\in \cW}F_{n_1}(S^{n_2})\wedge X(n_1,n_2)_+
\]
of the indicated $\cW^{\op}\times \cW$-diagram of orthogonal spectra.  Notice in particular that 
$\mS^{\cW}[F^{\cW}_{(d_1,d_2)}(*)]$ is naturally isomorphic to $F_{d_1}(S^{d_2})$. (A more explicit description of $\mS^{\cW}[X]$ is given in Proposition~\ref{prop:graded-SW-formula}.) For an orthogonal spectrum~$E$, let $\Omega^{\cW}(E)$ be the $\cW$-space defined by
\[
\Omega^{\cW}(E)(n_1,n_2)=\SpO(F_{n_1}(S^{n_2}),E),
\]
where the right hand side is the morphism space in the topological category $\SpO$.
By adjointness, the latter space can be identified with $\Omega^{n_2}(E_{n_1})$.
The functoriality of the construction assigns to a morphism $[m,\sigma_1,\sigma_2]\colon 
(m_1,m_2)\to (n_1,n_2)$ in $\cW$ the map 
$\Omega^{m_2}(E_{m_1})\to \Omega^{n_2}(E_{n_1})$ that takes 
$f\colon S^{m_2}\to E_{m_1}$ to the composition
\[
S^{n_2}\xr{\sigma_2^{-1}}S^{m_2}\wedge S^{m}\xr{f\wedge \id}E_{m_1}\wedge S^m\to E_{m_1+m}\xr{\sigma_1}E_{n_1}.
\]
In the next proposition we consider the absolute and positive stable model structures on $\SpO$ introduced in~\cite{MMSS}.

\begin{proposition}\label{prop:S-Omega-adjunction}
The functor $\mS^{\cW}$ is strong symmetric monoidal and the functor $\Omega^{\cW}$ is lax symmetric monoidal. These functors define a Quillen adjunction
\[
\mS^{\cW}\colon \Top^{\cW}\rightleftarrows \SpO:\!\Omega^{\cW}
\]
with respect to the (positive) $\cW$-model structure on $\Top^{\cW}$ and the (positive) stable model structure on $\SpO$.
\end{proposition}
\begin{proof}
It is clear from the formal properties of the coend construction that $\mS^{\cW}$ is left adjoint to $\Omega^{\cW}$ (see e.g.\ \cite[Section~2]{Mandell_M-equivariant-orthogonal} for a general discussion of this phenomenon in a based topological context). For this to be a Quillen adjunction it suffices to show that $\Omega^{\cW}$ preserves fibrations and acyclic fibrations and this is clear from the characterization of stable fibrations in \cite[Proposition~9.5]{MMSS}. The statements about monoidality 
follow formally from the fact that the functor $F_{-}(S^{-})$ in Lemma~\ref{lem:FWop-functor} is strong symmetric monoidal, cf.\ \cite[Proposition~2.14]{Mandell_M-equivariant-orthogonal}.
\end{proof}

In more detail, the monoidal structure map $\mS^{\cW}[X]\wedge
\mS^{\cW}[Y]\to \mS^{\cW}[X\wedge Y]$ is obtained by passage to coends from the natural maps
\[
\begin{aligned}
&\big(F_{m_1}(S^{m_2})\wedge X(m_1,m_2)_+ \big)\wedge 
\big(F_{n_1}(S^{n_2})\wedge Y(n_1,n_2)_+ \big) \\
&\xr{\cong} 
F_{m_1}(S^{m_2})\wedge  F_{n_1}(S^{n_2})\wedge\big(X(m_1,m_2)\times Y(n_1,n_2)\big)_+ \\
&\to \big(F_{m_1+n_1}(S^{m_2+n_2}) \wedge \big(X\boxtimes Y(m_1+n_1,m_2+n_2)\big)_+
\end{aligned}
\]
whereas the monoidal structure of $\Omega^{\cW}$ is given by the natural maps 
\[
\Omega^{m_2}(E_{m_1})\wedge \Omega^{n_2}(F_{n_1})\to \Omega^{m_2+n_2}(E_{m_1}\wedge
F_{n_1})\to \Omega^{m_2+n_2}((E\wedge F)_{m_1+n_1})
\]
where the first arrow takes a pair of maps to their smash product.

The functor $\mathbb S^{\cW}$, being a left Quillen functor, takes $\cW$-equivalences between cofibrant $\cW$-spaces to stable equivalences. It is important for our applications that 
$\mathbb S^{\cW}$ is homotopically well-behaved on a larger class of $\cW$-spaces.

\begin{definition}\label{def:SW-good}
A $\cW$-space $X$ is said to be \emph{$\mathbb S^{\cW}$-good} if there exists a cofibrant $\cW$-space $X'$ and a $\cW$-equivalence $X'\to X$ such that $\mathbb S^{\cW}[X']\to \mathbb S^{\cW}[X]$ is a stable equivalence.
\end{definition}
It is clear from the definition that cofibrant $\cW$-spaces are
$\mathbb S^{\cW}$-good. The terminal $\cW$-space $*$ is an example of
a $\cW$-space that is not $\mathbb S^{\cW}$-good. 
Using that $\mathbb S^{\cW}$ is a left Quillen functor we see that if
$X$ is $\mathbb S^{\cW}$-good and $Y\to X$ is any $\cW$-equivalence
with $Y$ cofibrant, then the induced map $\mathbb S^{\cW}[Y]\to
\mathbb S^{\cW}[X]$ is a stable equivalence. This in turn has the
following consequence.
\begin{proposition}\label{prop:SW-good-equivalence}
The functor $\mathbb S^{\cW}$ takes $\cW$-equivalences between $\cW$-spaces that are 
$\mathbb S^{\cW}$-good to stable equivalences. \qed 
\end{proposition}

In Appendix~\ref{app-sec:SW-goodness} we examine the notion of $\mS^{\cW}$-goodness in more detail and we formulate some convenient criteria which ensure that a $\cW$-space is $\mS^{\cW}$-good. 

\begin{remark}
The category $\cW$ was first considered in connection with orthogonal spectra by
Kro~\cite{Kro-orthogonal} who used $\cW$ to construct a convenient
fibrant replacement functor on $\SpO$.
\end{remark}

\subsection{Graded orthogonal spectra}\label{subsec:graded-orthogonal}
By a graded orthogonal spectrum  we understand a family of orthogonal spectra $E=\{E_{\{d\}}\colon d\in \mathbb Z\}$ and we write $\Grad_{\mathbb Z}\SpO$ for the category of graded orthogonal spectra in which a morphism $f\colon D\to E$ is a family of maps of orthogonal spectra $f_{\{d\}}\colon D_{\{d\}}\to E_{\{d\}}$ indexed by $d\in \mathbb Z$. The obvious ``graded smash product'' $D\wedge E$ with
\[
(D\wedge E)_{\{d\}}=\bigvee_{i+j=d} D_{\{i\}}\wedge E_{\{j\}}
\]
makes $\Grad_{\mathbb Z}\SpO$ a symmetric monoidal category with monoidal unit the graded orthogonal spectrum which is the sphere spectrum $\mathbb S$ in degree $0$ and the terminal spectrum $*$ in all other degrees. We use the term \emph{graded orthogonal ring spectrum} for a monoid in $\Grad_{\mathbb Z}\SpO$. If $R$ is a graded orthogonal ring spectrum, then $R_{\{0\}}$ is an ordinary orthogonal ring spectrum and each $R_{\{d\}}$ is an $R_{\{0\}}$-module. 
There are adjoint functors 
 $t\colon \Grad_{\mathbb Z}\SpO\rightleftarrows \SpO:\! c$ defined by 
 $t(E)=\bigvee_{d\in \mathbb Z}E_{\{d\}}$ and $c(Z)_{\{d\}}=Z$. The functor $t$ is strong symmetric monoidal and $c$ is lax symmetric monoidal.  
 
Now let us view $F_{n_1}(S^{n_2})$ as a graded orthogonal spectrum concentrated in degree 
 $n_2-n_1$. The functor $F_{-}(S^{-})$ from Lemma~\ref{lem:FWop-functor} then factors through a  strong symmetric monoidal functor 
$F_{-}(S^{-})\colon\cW^{\op}\to \Grad_{\mathbb Z}\SpO$ such that the composition with $t$ is the functor considered in the lemma. From this we obtain the horizontal 
adjunction $(\mathbb S^{\cW}_{\mathbb Z},\Omega^{\cW}_{\mathbb Z})$ in the commutative diagram of adjunctions
\[
\xymatrix{
\Top^{\cW}\ar@<.5ex>[rr]^{\mathbb S^{\cW}_{\mathbb Z}}  
\ar@<.5ex>[dr]^(.6){\!\mathbb S^{\cW}}& & \Grad_{\mathbb Z}\SpO
\ar@<.5ex>[ll]^{\Omega^{\cW}_{\mathbb Z}}\ar@<-.5ex>[dl]_t\\
& \SpO.\ar@<.5ex>[ul]^{\Omega^{\cW}}\ar@<-.5ex>[ur]_c &
}
\]
For a $\cW$-space $X$, we have $\mathbb S^{\cW}_{\mathbb Z}[X]_{\{d\}}=\mathbb S^{\cW}[X_{\{d\}}]$ where $X_{\{d\}}$ denotes the $\cW$-space which agrees with $X$ on the $d$th component 
$\cW_{\{d\}}$  and is the empty set otherwise. We note that $\mathbb S^{\cW}_{\mathbb Z}$ is strong symmetric monoidal and $\Omega^{\cW}_{\mathbb Z}$ is lax symmetric monoidal.  

\begin{proposition}\label{prop:graded-SW-formula}
For a $\cW$-space $X$ and $d\in \mZ$, we have 
\begin{equation}\label{eq:graded-formula}
\mathbb S^{\cW}[X_{\{d\}}]_n\iso 
\begin{cases}
X(n,d+n)_+\wedge_{O(d+n)}S^{d+n} & \text{if $d+n\geq 0$},\\
* & \text{if $d+n<0$}.
\end{cases}
\end{equation}
\end{proposition}
Here the notation indicates the coequalizer of the left $O(d+n)$-action on $S^{d+n}$ and the right 
$O(d+n)$-action on $X(n,d+n)$ obtained by letting $a\in O(d+n)$ act as $a^{-1}$ in the second variable. (This is the same as the orbit space of the diagonal left action). For $0<-d<n$ we again think of 
$\mR^{d+n}$ as the orthogonal complement of $\mR^{-d}$ included as the first $-d$ coordinates in $\mR^n$.

\begin{proof}
We first observe that the right hand side of  \eqref{eq:graded-formula} actually defines an orthogonal spectrum. The orthogonal group $O(n)$ acts via the left action on $X(n,d+n)$ in the first variable and the spectrum structure map
\[
X(n,d+n)_+\wedge_{O(d+n)}S^{d+n}\wedge S^1\to X(n+1,d+n+1)_+\wedge_{O(d+n+1)}S^{d+n+1}
\]
is induced by the morphism
\[
[1,\id_{\mR^{n+1}},\id_{\mR^{d+n+1}}]\colon (n,d+n)\to (n+1,d+n+1)
\]
and the canonical identification $S^{d+n}\wedge S^1=S^{d+n+1}$. For each $(n_1,n_2)$ in 
$\cW_{\{d\}}$ this spectrum receives a spectrum map from 
$F_{n_1}(S^{n_2})\wedge X(n_1,n_2)_+ $ induced by the obvious map of spaces
\[
S^{n_2}\wedge X(n_1,n_2)_+\to X(n_1,d+n_1)_+\wedge_{O(d+n_1)}S^{d+n_1}
\]
where $n_2=d+n_1$. By the universal property of the coend, these spectrum maps assemble to a map from $\mS^{\cW}[X_{\{d\}}]$ which gives the isomorphism in the proposition. The inverse is defined in spectrum degree $n$ by factoring the composition 
\[
X(n,d+n)_+\wedge S^{d+n}\to F_n(S^{d+n})_n\wedge X(n,d+n)_+\to \mS^{\cW}[X_{\{d\}}]_n
\]
over the coequalizer by the $O(d+n)$-action.
\end{proof}

The right adjoint $\Omega^{\cW}_{\mZ}$ takes a graded orthogonal spectrum $E$ to the $\cW$-space 
$\Omega^{\cW}_{\mZ}(E)$ whose restriction to $\cW_{\{d\}}$ equals the restriction of $\Omega^{\cW}(E_{\{d\}})$ to $\cW_{\{d\}}$.   

Using that the adjoint functors $\mathbb S^{\cW}_{\mathbb Z}$ and $\Omega^{\cW}_{\mZ}$ are (lax) symmetric monoidal, we get an induced adjunction between the  corresponding
categories of (commutative) monoids. In particular, a $\cW$-space monoid $M$ gives rise to the graded orthogonal ring spectrum $\mathbb S^{\cW}_{\mathbb Z}[M]$ with graded multiplication
\[
\mathbb S^{\cW}[M_{\{d\}}]\wedge \mathbb S^{\cW}[M_{\{e\}}]\xr{\cong} 
\mathbb S^{\cW}[M_{\{d\}}\boxtimes M_{\{e\}}]\to \mathbb S^{\cW}[M_{\{d+e\}}].
\]
With the explicit description in Proposition~\ref{prop:graded-SW-formula}, this multiplication takes the form
\[
\begin{aligned}
&\big(M(m,d+m)_+\wedge_{O(d+m)}S^{d+m}\big)\wedge\big(M(n,e+n)_+\wedge_{O(e+n)}S^{e+n}\big)\\
&\to M(m+n,d+m+e+n)_+\wedge_{O(d+m+e+n)}S^{d+m+e+n}\\
&=M(m+n,d+e+m+n)_+\wedge_{O(d+e+m+n)}S^{d+e+m+n}
\end{aligned}
\]
where the first map is given by the monoid structure of $M$ and the canonical identification of $S^{d+m}\wedge S^{e+n}$ with $S^{d+m+e+n}$. The last equality indicates that the action of $O(d+m+e+n)$ renders further permutations of the coordinates irrelevant.

\subsection{Grouplike commutative $\cW$-space monoids}\label{sec:grouplike}
We say that a commutative $\cW$-space monoid $M$ is \emph{grouplike} if the underlying $E_{\infty}$ space $M_{h\cW}$ is grouplike in the usual sense, that is, if $\pi_0(M_{h\cW})$ is a group. This can also be expressed in terms of the monoid unit $*\to M(0,0)$ and the map of $\cW\times\cW$-spaces \[
\mu\colon M(m_1,m_2)\times M(n_1,n_2)\to M((m_1,m_2)\oplus(n_1,n_2))
\] 
defining the multiplication. 

\begin{lemma}\label{lem:grouplike-criterion}
A commutative $\cW$-space monoid $M$ is grouplike if and only if for each element $x\in M(m_1,m_2)$ there exists an element $y\in M(n_1,n_2)$ such that $n_2-n_1=m_1-m_2$ and $\mu(x,y)$ belongs to the same path component as the image of the monoid unit $*\to M(0,0)$ under the canonical morphism
\[
[m_1+n_1,\id_{\mR^{m_1+n_1}},\id_{\mR^{m_2+n_2}}]\colon (0,0)\to (m_1+n_1,m_2+n_2).
\]
\end{lemma}
\begin{proof}
Let $\pi_0\cW$ be the category with the same objects as $\cW$ and morphisms the path components of the morphism spaces in $\cW$. Then it follows from the definition of $M_{h\cW}$ as the realization of a simplicial space that $\pi_0(M_{h\cW})$ can be identified with the colimit of the $\pi_0\cW$-diagram defined by the path components $\pi_0(M(n_1,n_2))$. Restricting to the $d$th component $\cW_{\{d\}}$, we thus get that 
\[
\pi_0(M_{h\cW_{\{d\}}})\cong \colim_{n} \pi_0(M(n,d+n))/\pi_0(O(n)\times O(d+n))
\]
where the colimit is over the ordered set of natural numbers $n$ such that $d+n\geq 0$. This easily gives the statement in the lemma.
\end{proof}

\begin{example}[Graded $\cW$-space units]
Let $R$ be a commutative orthogonal ring spectrum and assume for simplicity that $R$ is positive fibrant. We can then define the graded $\cW$-space units $\GL_1^{\cW}(R)$ by mimicking the analogous construction for symmetric spectra in \cite[Section~4]{Sagave-S_diagram}: Every path component in the space $\Omega^{\cW}(R)(n_1,n_2)$ represents an element in the graded ring of homotopy groups 
$\pi_*(R)$ (see below), and we define $\GL_1^{\cW}(R)(n_1,n_2)$ to be the union of the path components that represent graded units. This defines a sub $\cW$-space $\GL_1^{\cW}(R)$ of $\Omega^{\cW}(R)$, and it is clear that $\GL_1^{\cW}(R)$ inherits the structure of a commutative $\cW$-space monoid which is grouplike by Lemma~\ref{lem:grouplike-criterion}.
\end{example}

In order to  analyze orthogonal ring spectra obtained from grouplike commutative $\cW$-space monoids, we review some facts about the homotopy groups of orthogonal spectra. Recall  that the $n$th homotopy group of an orthogonal spectrum $E$ is defined by $\pi_n(E)=\colim_k\pi_{n+k}(E_k)$, where the colimit is over the  homomorphisms
\[
[S^{n+k},E_k]_*\xr{(-)\wedge \id_{S^1}} [S^{n+k}\wedge S^1,E_k\wedge S^1]_*\to [S^{n+k+1},E_{k+1}]_*
\]
specified by the spectrum structure maps. Here we require that $n+k\geq 2$ in order for this to be a sequence of abelian groups. With this definition, each class in $\pi_{n_2}(E_{n_1})$ defines an element in $\pi_{n_2-n_1}(E)$ by passage to the colimit. We write $\pi_*(E)$ for the resulting $\mZ$-graded abelian group.

There is an exterior pairing of homotopy groups which for orthogonal spectra $E$ and $E'$ takes the form
\[
\pi_*(E)\otimes \pi_*(E')\to \pi_*(E\wedge E'),\quad \alpha\otimes \beta\mapsto \alpha\cdot \beta
\]
where for representatives $[f]\in \pi_{m_2}(E_{m_1})$ and $[g]\in\pi_{n_2}(E'_{n_1})$, we set
\[
[f]\cdot[g]=(-1)^{m_1(n_2-n_1)}[S^{m_2}\wedge S^{n_2}\xr{f\wedge g} E_{m_1}\wedge E'_{n_1}
\to (E\wedge E')_{m_1+n_1}],
\]
see \cite{Schwede-SymSp} and \cite[Section~4]{Sagave-S_diagram}. In non-negative degrees the product can be realized on representatives without the sign by writing $f\colon S^{m+k}\to E_k$ and 
${g\colon S^{n+l}\to E'_l}$, and then precomposing with the map
$S^{m+n+k+l}\to S^{m+k+n+l}$ induced by the obvious block permutation. The formal properties of the exterior product can be summarized as follows.

\begin{proposition}
The canonical homomorphism $\mZ\cong \pi_0(\mS)\to \pi_*(\mS)$ and the exterior product give $\pi_*(-)$ the structure of a lax symmetric monoidal functor from the category of orthogonal spectra equipped with the smash product to the category of graded abelian groups equipped with the graded tensor product and the usual sign convention for the symmetric monoidal structure. \qed
\end{proposition}

It follows that if $R$ is a graded orthogonal ring spectrum with graded levels $R_{\{n\}}$, then~$\pi_*(R)$ is a bigraded ring with $(m,n)$th term $\pi_m(R_{\{n\}})$. If $R$ is commutative, then~$\pi_*(R)$ is graded commutative in the sense that $\alpha\cdot\beta=(-1)^{|\alpha||\beta|}  \beta\cdot \alpha$, where $|\alpha|$ and $|\beta|$ are the internal degrees in the respective graded abelian groups $\pi_*(R_{\{n\}})$. 

We also recall that if $n=n_2-n_1$, then the free orthogonal spectrum $F_{n_1}(S^{n_2})$ represents the suspension $\Sigma^n\mS$ as an object in the stable homotopy category:
\[
F_{n_1}(S^{n_2})\xr{\sim} \Omega^{n_1}(S^{n_1}\wedge F_{n_1}(S^{n_2}))
\xr{\sim}  \Omega^{n_1}(S^{n_2}\wedge F_{n_1}(S^{n_1}))\xr{\sim}
\Omega^{n_1}(S^{n_2}\wedge \mathbb S),
\]
where the last map is induced by the canonical stable equivalence $F_{n_1}(S^{n_1})\to \mathbb S$, cf.\ 
\cite[Lemma~8.6]{MMSS}. Let $[n_1,n_2]\in \pi_n(F_{n_1}(S^{n_2}))$ be the generator represented by the obvious map $S^{n_2}\to O(n_1)_+\wedge S^{n_2}$ taking $x$ to $(\id,x)$. Then the external pairing with $[n_1,n_2]$ defines an isomorphism of graded abelian groups
\[
\Sigma^n\pi_*(X)\to \pi_*(F_{n_1}(S^{n_2})\wedge X), \quad \alpha\mapsto [n_1,n_2]\cdot \alpha
\]
for every orthogonal spectrum $X$. Here and in the following, the notation $\Sigma^n\pi_*(X)$ indicates the graded abelian group with $k$th term $\Sigma^n\pi_k(X)=\pi_{k-n}(X)$. 

Now let $M$ be a grouplike commutative $\cW$-space monoid which we keep fixed for the rest of this section. We let $d\geq 0$ denote the \emph{periodicity degree} of $M$, that is, $d$ generates the image of the group homomorphism $\pi_0(M_{h\cW})\to\pi_0(B\cW)\cong \mZ$ induced by the projection of $M$ onto the terminal $\cW$-space. An element $u\in M(n_1,n_2)$ with $n=n_2-n_1$ gives rise to a map of $\cW$-spaces $\bar u\colon F_{(n_1,n_2)}^{\cW}(*)\to M_{\{n\}}$ and applying $\mS^{\cW}$ we get a map of orthogonal spectra $F_{n_1}(S^{n_2})=\mS^{\cW}[F_{(n_1,n_2)}^{\cW}(*)]\to \mS^{\cW}[M_{\{n\}}]$. This in turn amounts to a map of spaces $S^{n_2}\to \mS^{\cW}[M_{\{n\}}]_{n_1}$ and hence determines an element $[u]\in \pi_n(\mS^{\cW}[M_{\{n\}}])$. By the explicit description in 
Proposition~\ref{prop:graded-SW-formula}, this element is represented by the map $S^{n_2}\to M(n_1,n_2)_+\wedge S^{n_2}$ taking $x$ to $(u,x)$.  

\begin{lemma}\label{lem:units-in-SWM}
For a grouplike commutative $\cW$-space monoid $M$, every element $u\in M(n_1,n_2)$ with $n=n_2-n_1$ determines a unit $[u]\in \pi_n(\mS^{\cW}[M_{\{n\}}])$ in the bigraded ring of homotopy groups $\pi_*(\mS^{\cW}[M])$. 
\end{lemma}
\begin{proof}
It follows from Lemma~\ref{lem:grouplike-criterion} that there exists an element $v\in M(m_1,m_2)$ such that $m_2-m_1=-n$ and the equalities $[u][v]=1$ and $[v][u]=1$ hold in $\pi_*(\mS^{\cW}[M])$. This uses the explicit description of the multiplication and the fact that if a commutative orthogonal ring spectrum has a unit in odd degree, then the graded ring of homotopy groups is an $\mF_2$-algebra.
\end{proof}
In subsequent sections we will be interested in periodic ring spectra arising from the following construction.
\begin{definition}\label{def:RP}
Let $R$ be an (ungraded) commutative orthogonal ring spectrum and let $\mS^{\cW}[M_{\{0\}}]\to R$ be a map of orthogonal ring
  spectra. Then we let $RP$ be the commutative graded orthogonal ring
  spectrum with $n$th graded level
  \[
  RP_{\{n\}}=\mS^{\cW}[M_{\{n\}}]\wedge _{\mS^{\cW}[M_{\{0\}}]}R
  \]
and multiplication inherited from the multiplicative structures of $\mS^{\cW}[M]$ and $R$.
\end{definition}

By construction, $RP$ depends on the grouplike commutative $\cW$-space monoid $M$ even though this is implicit in the notation. The notation is motivated by Proposition~\ref{prop:SW-M-equivalence} below which shows that  
$t(RP)$ is stably equivalent to $\bigvee_{n\in d\mZ}\Sigma^nR$, where $d$ is the periodicity degree introduced above. 
Given an element $u\in M(n_1,n_2)$ with $n=n_2-n_1$, we also use the notation $[u]\in \pi_n(RP_{\{n\}})$ to denote the unit in $\pi_*(RP)$ represented by the composition 
\[
F_{n_1}(S^{n_2})\to \mS^{\cW}[M_{\{n\}}]\to RP_{\{n\}}. 
\]

\begin{proposition}\label{prop:SW-M-equivalence}
Every element $u\in M(n_1,n_2)$ gives rise to a stable equivalence of $R$-modules $\Sigma^nR\simeq F_{n_1}(S^{n_2})\wedge R\xr{\sim} RP_{\{n\}}$. 
\end{proposition}
\begin{proof}
The stable equivalence is defined as the composition 
\[
F_{n_1}(S^{n_2})\wedge R\to RP_{\{n\}}\wedge R\to RP_{\{n\}}
\]
where the last map is the multiplication in $RP$. On the level of homotopy groups this induces multiplication by $[u]$,
\[
\Sigma^n\pi_*(R)\xr{\cong} \pi_*(F_{n_1}(S^{n_2})\wedge R)\to \pi_*(RP_{\{n\}}),\quad \alpha\mapsto [u]\cdot \alpha,
\]
and the result follows because $[u]$ is a unit. 
\end{proof}

Notice, that for $R=\mS^{\cW}[M_{\{0\}}]$ the proposition gives  a stable equivalence 
\[
\Sigma^n\mS^{\cW}[M_{\{0\}}]\simeq F_{n_1}(S^{n_2})\wedge \mS^{\cW}[M_{\{0\}}]\xr{\sim} \mS^{\cW}[M_{\{n\}}].
\]

\section{Graded Grassmannians and periodic cobordism spectra}\label{sec:graded-Grass} In this section we show how the classical Stiefel manifolds fit in the setting of \mbox{$\cW$-spaces}. Using this, we introduce the oriented and unoriented graded Grassmannians in our context, and we show how the periodic real cobordism spectra $\MOP$ and $\MSOPeven$ arise from commutative $\cW$-space monoids.

\subsection{The Stiefel $\cW$-space $V$}\label{subsec:V-W-space} In the following $\oplus^{\infty}\mR^n$ denotes the infinite direct sum $\oplus_{i=1}^{\infty}\mR^n$, and we write $\cV(\mR^m,\oplus^{\infty}\mR^n)$ for the space of linear isometries from $\mR^m$ to $\oplus^{\infty}\mR^n$. (This is a slight abuse of notation since $\oplus^{\infty}\mR^n$ is not an object of the category $\cV$.) We shall define a ``Stiefel'' $\cW$-space $V$ by
\[ V\colon \cW\to\Top,\quad V(n_1,n_2)=\cV(\mR^{n_2},\oplus^{\infty}\mR^{n_1}).
\] Thus, $V(n_1,n_2)$ can be identified with the space $V_{n_2}(\oplus^{\infty}\mR^{n_1})$ of orthogonal $n_2$-frames in $\oplus^{\infty}\mR^{n_1}$, topologized as the colimit of the finite dimensional Stiefel manifolds $V_{n_2}(\oplus^k\mR^{n_1})$. This is the empty set if and only if $n_1=0$ and $n_2>0$, and is otherwise contractible. In order to specify the action on morphisms, we fix a choice of isometries $i_m\colon \mR^m\to \oplus^{\infty}\mR^m$ for $m\geq 0$ by including $\mR^m$ as the first summand.  Given a morphism $[m,\sigma_1,\sigma_2]\colon (m_1,m_2)\to (n_1,n_2)$ in $\cW$, the induced map
\[ [m,\sigma_1,\sigma_2]_*\colon V(m_1,m_2)\to V(n_1,n_2)
\] then takes an element $\bld v$ in the domain to the isometry determined by the commutativity of the diagram
\[ \xymatrix@-.5pc{ \mR^{n_2} \ar[rr]^-{[m,\sigma_1,\sigma_2]_*(\bld v) }& & \oplus^{\infty}\mR^{n_1}\\ \mR^{m_2}\oplus\mR^m \ar[r]^-{\bld v\oplus i_m} \ar[u]^{\sigma_2}& (\oplus^{\infty}\mR^{m_1})\oplus(\oplus^{\infty}\mR^m) \ar[r]^-{\cong} &\oplus^{\infty}(\mR^{m_1}\oplus\mR^m) \ar[u]_{\oplus^{\infty}\sigma_1}.  }
\] Here the isomorphism on the bottom right is the obvious order preserving permutation of the coordinates.  

\begin{lemma}\label{lem:VhW-BW-equivalence}
The projection of $V$ onto the terminal $\cW$-space is a positive acyclic $\cW$-fibration and thus induces a weak homotopy equivalence $V_{h\cW}\to B\cW$.\qed
\end{lemma}

Next we consider the natural map of ($\cW\times\cW$)-spaces
\[ \mu\colon V(m_1,m_2)\times V(n_1,n_2)\to V((m_1,m_2)\oplus (n_1,n_2))
\] sending a pair $\bld v\in V(m_1,m_2)$ and $\bld w\in V(m_1,m_2)$ to the isometry
\[ \mR^{m_2}\oplus \mR^{n_2} \xr{\bld v\oplus\bld w} (\oplus^{\infty}\mR^{m_1}) \oplus(\oplus^{\infty}\mR^{n_1}) \xr{\cong} \oplus^{\infty}(\mR^{m_1}\oplus\mR^{n_1})
\] where we use the same order preserving permutation of the coordinates.  By the universal property of the $\boxtimes$-product, this in turn gives rise to a map of $\cW$-spaces $\mu\colon V\boxtimes V\to V$. Checking from the definitions, we get the next result.

\begin{proposition} The multiplication $\mu\colon V\boxtimes V\to V$ makes $V$ a grouplike commutative $\cW$-space monoid with unit the unique map of $\cW$-spaces $U^{\cW}\to V$. \qed
\end{proposition}

\subsection{The unoriented periodic cobordism spectrum $MOP$} We proceed to show that the corresponding commutative orthogonal ring spectrum $\mS^{\cW}[V]$ is a model of the unoriented periodic cobordism spectrum. Using the description from Proposition~\ref{prop:graded-SW-formula}, we get that
\[ \mS^{\cW}[V_{\{d\}}]_n=
\begin{cases} V_{d+n}(\oplus^{\infty}\mR^n)_+\wedge_{O(d+n)}S^{d+n},&\text{ if $d+n\geq 0$},\\ *,& \text{ if $d+n<0$}.
\end{cases}
\] In this expression, the $O(n)$-action on the right hand side is induced from the diagonal left $O(n)$-action on $\oplus^{\infty}\mR^n$.

For $d=0$ we have $\mS^{\cW}[V_{\{0\}}]_n=V_n(\oplus^{\infty}\mR^n)_+\wedge_{O(n)}S^n$, which can be identified with the Thom space of the canonical vector bundle over the Grassmannian $\Gr_n(\oplus^{\infty}\mR^n)$ of $n$-planes in $\oplus^{\infty}\mR^n$. It follows that $\mS^{\cW}[V_{\{0\}}]$ is a model of the unoriented real cobordism spectrum which justifies introducing the standard notation
\[ \MO=\mS^{\cW}[V_{\{0\}}]
\] in our context. Since $V$ is grouplike, setting $R = \mS^{\cW}[V_{\{0\}}]$ in Proposition~\ref{prop:SW-M-equivalence} results in a stable equivalence $ \mS^{\cW}[V]\simeq \bigvee_{n\in \mZ}\Sigma^n\MO $. This justifies introducing the standard notation
\[ \MOP=\mS^{\cW}[V]
\] for the unoriented periodic cobordism spectrum.

\subsection{Graded Grassmannians} Before getting to the graded Grassmannians we are after, it is illuminating to discuss the passage from $\cW$-spaces to graded \mbox{$\cV$-spaces} in general. Let $\rho\colon \cW\to\cV$ be the functor that projects on the first variable, $\rho(n_1,n_2)=n_1$, and
\[ \rho\big([m,\sigma_1,\sigma_2]\colon (m_1,m_2)\to(n_1,n_2)\big) =\big([m,\sigma_1]\colon m_1\to n_1\big),
\] where we identify the morphism space $\cV(m_1,n_1)$ with $\cO(m_1\oplus m,n_1)/O(m)$. The functor $\rho$ gives rise to an adjunction $\rho_*\colon\Top^{\cW}\rightleftarrows \Top^{\cV}:\!\rho^*$, in which $\rho_*$ denotes the left Kan extension and $\rho^*$ pulls back a $\cV$-space to a $\cW$-space via $\rho$. We shall be interested in a graded refinement of this and write $\Grad_{\mZ}\Top^{\cV}$ for the category of $\mZ$-graded $\cV$-spaces $X=\{X_{\{d\}}\colon d\in \mZ\}$, equipped with the symmetric monoidal ``graded'' $\boxtimes$-product
\[ (X\boxtimes Y)_{\{d\}}=\coprod_{i+j=d}X_{\{i\}}\boxtimes Y_{\{j\}}.
\] The graded $\boxtimes$-product has as its monoidal unit the graded $\cV$-space which is $U^{\cV}$ (the monoidal unit for $\Top^{\cV}$) in degree $0$ and the initial (that is, empty) $\cV$-space in all other degrees. As in the discussion of graded orthogonal spectra (cf.\ Section~\ref{subsec:graded-orthogonal}), we have adjoint functors $t\colon \Grad_{\mZ}\Top^{\cV}\rightleftarrows \Top^{\cV}\!:\!c$, given by $t(X)=\coprod_{d\in \mZ}X_{\{d\}}$ and $c(Y)_{\{d\}}=Y$ for all $d\in \mZ$. 

Our aim is to lift $\rho_*$ to a functor $\rho_*^{\mZ}\colon\Top^{\cW}\to \Grad_{\mZ}\Top^{\cV}$ and for this we set 
\begin{equation}\label{eq:rhoZ} 
\rho^{\mZ}_*(X)_{\{d\}}(n)=
\begin{cases} X(n,d+n)/O(d+n),& \text{if $d+n\geq 0$}\\ \emptyset,& \text{if $d+n<0$}.
\end{cases}
\end{equation} We often write $\rho^d_*(X)=\rho^{\mZ}_*(X)_{\{d\}}$. Given a morphism $[m',\sigma]\colon m\to n$ in $\cV$ with $d+m\geq 0$, the induced map $\rho^d_*(X)([m',\sigma])$ is defined by passage to orbit spaces
\[ X(m,d+m)/O(d+m)\to X(n,d+n)/O(d+n),\qquad x\mapsto [m',\sigma_1,\sigma_2]_*(x),
\] where $[m',\sigma_1,\sigma_2]\colon (m,d+m)\to (n,d+n)$ denotes any morphism in $\cW$ that projects to $[m',\sigma]$ under $\rho$. The next lemma can be checked from the definitions.

\begin{lemma}\label{lem:rho-rhoZ-adjunctions} There is a commutative diagram of adjunctions
\[ 
\xymatrix{
\Top^{\cW} \ar@<.5ex>[rr]^-{\rho_*^{\mZ}} \ar@<.5ex>[dr]^-{\rho_*}& & \ar@<.5ex>[ll]^-{\rho_{\mZ}^*}  \Grad_{\mZ}\Top^{\cV}\ar@<-.5ex>[dl]_-t\\
& \Top^{\cV} \ar@<.5ex>[ul]^-{\rho^*}  \ar@<-.5ex>[ur]_-c& 
}
\] 
where the right adjoint $\rho_{\mZ}^*$ takes a graded $\cV$-space $X$ to the $\cW$-space $\rho^*_{\mZ}(X)$ whose restriction to $\cW_{\{d\}}$ equals $\rho^*(X_{\{d\}})$. Here all left adjoints are strong symmetric monoidal and all right adjoins are lax symmetric monoidal. 
\qed
\end{lemma} 

Recall the definition of the Stiefel $\cW$-space $V$ from Section~\ref{subsec:V-W-space}.

\begin{definition} The (unoriented) graded Grassmannian $\Gr$ is the commutative graded $\cV$-space monoid defined by $\Gr=\rho^{\mZ}_*(V)$.
\end{definition}

Thus, writing $\Gr_{d+n}(\oplus^{\infty}\mR^n)$ for the Grassmannian of $(d+n)$-planes in $\oplus^{\infty}\mR^n$, we have that
\[ \Gr_{\{d\}}(n)=
\begin{cases} \Gr_{d+n}(\oplus^{\infty}\mR^n), & \text{if $d+n\geq 0$} \\ \emptyset, & \text{if $d+n<0$}.
\end{cases}
\] 
with multiplication inherited from $V$.  It will be convenient to use the same notation $\Gr$ for the corresponding (ungraded) $\cV$-space $t(\Gr)$. The meaning will always be clear from the context.

In the next lemma we use the term ``$E_{\infty}$ space'' to mean an algebra over the topological Barratt--Eccles operad $\cE$  (see e.g.\ \cite{May_geometry}*{Construction~15.1}).

\begin{lemma}\label{lem:Gr_hV_identification}
There is a chain of weak equivalences of $E_{\infty}$ spaces $\Gr_{h\cV} \simeq B\cW$. 
\end{lemma}
\begin{proof}
Using the definition of $\Gr$ 
as a left Kan extension and writing $\rho_*^h(V)$ for the corresponding homotopy left Kan extension, we consider the following chain of maps
\[
\Gr_{h\cV}=\rho_*(V)_{h\cV}\ot\rho_*^h(V)_{h\cV}\to V_{h\cW}\to B\cW.
\]
Here the map $\rho_*^h(V)_{h\cV}\to V_{h\cW}$ is the canonical weak homotopy equivalence defined as in 
\cite[Theorem 5.5]{Hollender-V_modules} and the last map is the weak homotopy equivalence in Lemma~\ref{lem:VhW-BW-equivalence}. 
It remains to show that the map $\rho_*^h(V)\to 
\rho_*(V)$ induces a weak homotopy equivalence after passing to homotopy colimits. This map is in fact a level-wise equivalence as can be deduced from the identification $\rho_*^h(V)=\rho_*(\overline V)$. Indeed, under this identification, the map in question is induced by the level-wise equivalence $\epsilon\colon\overline V\to V$ and it follows from the explicit description in  \eqref{eq:rhoZ} that $\rho_*$ preserves level-wise equivalences between \mbox{$\cW$-spaces} that are, in the sense of 
 Definition~\ref{def:O-cof-second-var},  $O$-cofibrant in the second variable. Finally we remark that since the functor $\rho\colon \cW\to\cV$ is strict symmetric monoidal, the arguments in \cite[Section~6]{Schlichtkrull-Thom_symmetric} apply to show that the weak equivalences defined above are maps of $\cE$-algebras.
\end{proof}

\subsection{Oriented graded Grassmannian}\label{sec:oriented-graded-grass}
We begin by realizing the first Postnikov section of $B\cW$ in terms of a symmetric monoidal determinant functor. Let us view the set $\mN$ of natural numbers $n\geq 0$ as a permutative category with only identity morphisms and monoidal structure given by addition. Let us further view the group with two elements $\{\pm 1\}$ as a permutative category with a single object. We write $\red \mN$ for the permutative category whose underlying monoidal category is the product category $\mN\times \{\pm1\}$ and whose symmetric structure is specified by the isomorphisms $(-1)^{mn}\colon m\oplus n\to n\oplus m$. With this definition, the determinant function induces a symmetric monoidal functor $\Det\colon \cO\to\red\mN$. Applying Quillen's categorical localization construction to 
$\mN$ and $\red\mN$, we get the categories 
$\mN^{-1}\mN$ and $\red\mN^{-1}\red\mN$. Clearly $\mN^{-1}\mN$ decomposes as the coproduct of the full subcategories $(\mN^{-1}\mN)_{\{d\}}$ with objects $(n_1,n_2)$ such that $n_2-n_1=d$. Each of these subcategories is isomorphic to the ordered set of natural numbers.

Now let $S\cW$ be the subcategory of $\cW$ defined by the pullback diagram
\[
\xymatrix@-1pt{
S\cW \ar[r]\ar[d] & \cW\ar[d]^{\Det}\\
\mN^{-1}\mN \ar[r] &\red\mN^{-1}\red\mN .
}
\]
Thus, $S\cW$ has the same objects as $\cW$ and 
\[ 
S\cW((m_1,m_2),(n_1,n_2))=\big\{[m,\sigma_1,\sigma_2]\colon \det(\sigma_1)=\det(\sigma_2)\big\}
\]  
where $n_1=m_1+m$ and $n_2=m_2+m$. Again we have a decomposition of $S\cW$ as the coproduct of the full subcategories $S\cW_{\{d\}}$ with objects $(n_1,n_2)$ such that $n_2-n_1=d$. 
Applying a suitable Grothendieck construction as in Section \ref{subsec:Grothendieck-construction}, one can check that the classifying space of $S\cW$ is a model of $B\mathit{SO}\times \mZ$. Just as the category $\mN^{-1}\mN$ is not a monoidal subcategory of $\red\mN^{-1}\red\mN$, the category $S\cW$ is also not a monoidal subcategory of $\cW$. For this reason we often restrict to the full subcategory $S\cW_{\even}$ defined as the union of the subcategories $S\cW_{\{d\}}$ with $d$ even. It is easy to check that $S\cW_{\even}$ is indeed a permutative  subcategory of $\cW$.

Let $\iota \colon S\cW \to \cW$ be the inclusion and consider the corresponding adjunction 
$\iota_* \colon \Top^{S\cW} \rightleftarrows \Top^{\cW} \colon \iota^*$  defined by left Kan extension and restriction along
$\iota$. 

\begin{lemma}\label{lem:i*-explicit}
For an $S\cW$-space $X$, there are homeomorphisms
\[
\iota_*(X)(n_1,n_2)\cong
O(n_1,n_2)\times_{\tilde O(n_1,n_2)}X(n_1,n_2),
\] 
where $O(n_1,n_2)=O(n_1)\times O(n_2)$ and $\tilde O(n_1,n_2)$ denotes the subgroup given by the pairs $(a_1,a_2)$ such that $\det a_1=\det a_2$.
\begin{proof}
We give the expression on the right hand side the structure of a $\cW$-space by assigning to a morphism $[m,\sigma_1,\sigma_2]\colon(m_1,m_2)\to(n_1,n_2)$ in $\cW$ the map
\[ 
O(m_1,m_2)\times_{\tilde O(m_1,m_2)}X(m_1,m_2)\to O(n_1,n_2)\times_{\tilde O(n_1,n_2)}X(n_1,n_2)
\] 
taking $\big((a_1,a_2),x\big)$ to $\big((\sigma_1(a_1\oplus \id), \sigma_2(a_2\oplus \id)), [m,\id,\id]_*(x)\big)$. It is easy to check that the $\cW$-space so defined has the universal property of a left Kan extension.
\end{proof}
\end{lemma}

\begin{definition} The \emph{oriented Stiefel $\cW$-space} is the $\cW$-space $\tilde V = \iota_*(\iota^*(V))$ and the \emph{oriented graded Grassmannian} is the graded $\cV$-space $\tilde\Gr = \rho^{\mZ}_*(\tilde V)$. Using the inclusion $\iota_{\even}\colon S\cW_{\even} \to \cW$ instead of $\iota$ leads to the \emph{oriented evenly graded Stiefel $\cW$-space}~$\tilde V_{\even}$ and the \emph{oriented evenly graded Grassmannian} $\tilde\Gr_{\even}$.
\end{definition}

It follows from the fact that $\iota_{\even}$ is a strict symmetric  monoidal functor that $\tilde V_{\even}$ is a commutative $\cW$-space monoid and that $\tilde\Gr_{\even}$ is a commutative graded $\cV$-space monoid. Using 
Lemma~\ref{lem:i*-explicit} we get the following explicit description of $\tilde V$:
\[ 
\tilde V(n_1,n_2)\iso
\begin{cases} O(n_2)\times_{SO(n_2)}V_{n_2}(\oplus^{\infty}\mR^{n_1}), & \textrm{ if } n_2>0\\ O(n_1)/SO(n_1), & \textrm{ if } n_2=0.
\end{cases}
\]
If $n_1>0$, then $V_{n_2}(\oplus^{\infty}\mR^{n_1})$ is contractible and $\tilde V(n_1,n_2)$ decomposes in two contractible path components. This gives the next result.
\begin{proposition} The $\cW$-space $V$ is positive fibrant and the adjunction counit $\tilde V\to V$ is a positive fibration.\qed
\end{proposition} Let us write $\tilde\Gr_m(\oplus^{\infty}\mR^n)$ for the Grassmannian of oriented $m$-planes in $\oplus^{\infty}\mR^n$. This can be identified with the orbit space $V_m(\oplus^{\infty}\mR^n)/SO(m)$. Applying the general formula for $\rho^{\mZ}_*$ in~\eqref{eq:rhoZ} to the description of $\tilde V$ given above, we get the following identification of the oriented graded Grassmanian: 
\[
 \tilde\Gr_{\{d\}}(n) \iso
\begin{cases} \tilde\Gr_{d+n}(\oplus^{\infty}\mR^n), &\text{if $d+n>0$}\\ O(n)/SO(n),& \text{if $d+n=0$}\\ \emptyset ,& \text{if $d+n<0$}.
\end{cases}
\]
As for $\Gr$, we sometimes use the same notation $\tilde\Gr$ for the corresponding (ungraded) $\cV$-space 
$t(\tilde\Gr)$. With these conventions, the adjunction counit $\tilde V  \to V$ and the unit of the adjunction $(\rho_*,\rho^*)$ give rise to a square of $\cW$-spaces
\begin{equation}\label{eq:grassmannians-stiefel-pullback} 
\xymatrix@-.5pt{ 
\tilde V\ar[r] \ar[d]_-{\tilde p}& V\ar[d]^-p\\ \rho^*\tilde\Gr \ar[r] & \rho^*\Gr.  
}
\end{equation} 
that can be shown to be a pullback diagram.

Now consider the commutative graded orthogonal ring spectrum $\mS^{\cW}[\tilde V_{\even}]$. By Proposition~\ref{prop:graded-SW-formula}, we have $\mS^{\cW}[\tilde V_{\{0\}}]_n\iso\tilde V(n,n)_+\wedge_{O(n)}S^n$, which can be identified with the Thom space of the canonical vector bundle over $\tilde\Gr_n(\oplus^{\infty}\mR^n)$. Hence $\mS^{\cW}[\tilde V_{\{0\}}]$ is a model of the oriented real cobordism spectrum and consequently we write $\MSO=\mS^{\cW}[\tilde V_{\{0\}}]$. Since $\tilde V_{\even}$ is grouplike, it follows from Proposition~\ref{prop:SW-M-equivalence} that there is a stable equivalence $\mS^{\cW}[\tilde V_{\even}]\simeq \bigvee_{n\in\mZ}\Sigma^{2n}\MSO$. Thus, $\mS^{\cW}[\tilde V_{\even}]$ is a model of the 2-periodic real oriented cobordism spectrum, and motivated by this we introduce the notation
\[ 
\MSOPeven=\mS^{\cW}[\tilde V_{\even}].
\]

\section{The graded Thom spectrum functor}\label{sec:graded-Thom-functor}
In the following we shall consider the over-category $\Top^{\cV}/\Gr$ of $\cV$-spaces over the Grassmannian $\cV$-space $\Gr$.  
Since every $\cV$-space over $\Gr$ inherits a grading from $\Gr$, the category $\Top^{\cV}/\Gr$ is isomorphic to the category $\Grad_{\mZ}\Top^{\cV}/\Gr$ of graded $\cV$-space over (the graded version of) $\Gr$. Hence the distinction between graded and ungraded $\cV$-spaces becomes immaterial in this context.  
The functor $\rho$ gives rise to a pair of adjunctions
\[ 
\xymatrix{ \Top^{\cW}/V \ar@<0.5ex>[r] & \Top^{\cW}/\rho^*\Gr\ar@<0.5ex>[l] \ar@<0.5ex>[r]^-{\rho_*} & \Top^{\cV}/\Gr\ar@<0.5ex>[l]^-{\rho^*} }
\] 
in which the first adjunction is given by post composition with and pullback along the adjunction unit $p\colon V\to \rho^*\Gr$. 
The category $\Top^{\cW}$ has the $\cW$-model structure from Section~\ref{sec:W-homotopy-theory} and Lind~\cite{Lind-diagram} has introduced an analogous $\cV$-model structure on $\Top^{\cV}$ whose weak equivalences are the $\cV$-equivalences, that is, the maps $X \to Y$ that induce weak homotopy equivalences $X_{h\cV}\to Y_{h\cV}$.
It is clear that the above adjunctions become Quillen adjunctions when we equip the participating over-categories with the induced model structures, cf.\ \cite[Theorem~7.6.4]{Hirschhorn-model}. Arguing as in the proof of 
\cite[Proposition~13.4]{Sagave-S_diagram}, we get the next result.

\begin{proposition}\label{prop:rho-p-Quillen-equivalence} The composite Quillen adjunction
\[ 
\xymatrix{ \rho_*\colon \Top^{\cW}/V \ar@<0.5ex>[r] & \ar@<0.5ex>[l] \Top^{\cV}/\Gr \colon p^* }
\] 
is a Quillen equivalence with respect to the $\cW$- and $\cV$-model structures.\qed 
\end{proposition}

We now consider graded Thom spectra arising from $\cV$-spaces over $\Gr$, which we define using the composite functor
\begin{equation}\label{eq:V-space-Thom}
T\colon \Top^{\cV}/\Gr\xr{p^*} \Top^{\cW}/V\xr{\mS^{\cW}} \SpO/\MOP
\end{equation} 
where $p^*$ is the right adjoint in the Quillen equivalence from Proposition~\ref{prop:rho-p-Quillen-equivalence}. In more detail, $p^*$ takes a map of $\cV$-spaces $f\colon X\to\Gr$ to the pullback $p^*_f(X)$ in
\[ 
\xymatrix@-.5pc{ p^*_f(X) \ar[rr] \ar[d]&& V\ar[d]^{p}\\ \rho^*X\ar[rr]^{\rho^*\!f} && \rho^*\Gr, }
\] 
and we set $T(f)=\mS^{\cW}[p^*_f(X)]$. Let us momentarily restrict to maps of $\cV$-spaces $f\colon X\to \Gr_{\{0\}}$ over the degree 0 part of $\Gr$. Then $p^*_f(X)(n,n)$ is obtained by pulling back the principal $O(n)$-bundle $V_n(\oplus^{\infty}\mR^n)\to \Gr_n(\oplus^{\infty}\mR^n)$ to $X(n)$, which implies that there is a pullback diagram of the associated vector bundles
\[ 
\xymatrix@-1pc{ p^*_f(X)(n,n)\times_{O(n)}\mR^n \ar[r] \ar[d]& V_n(\oplus^{\infty}\mR^n)\times_{O(n)}\mR^n\ar[d]\\ X(n) \ar[r]^-{f_n} & \Gr_n(\oplus^{\infty}\mR^n).  }
\] 
Hence we see from the explicit description in Proposition~\ref{prop:graded-SW-formula} that $T(f)_n$ is homeomorphic to the Thom space of the vector bundle classified by $f_n$. We conclude that applying the functor $T$ to $\cV$-spaces over $\Gr_{\{0\}}$ we get a model of the Thom spectrum for the corresponding sequence of vector bundles. The advantage of the current approach is that it extends to the graded setting and provides us with a lax symmetric monoidal graded Thom spectrum functor.

\begin{proposition}\label{prop:T-lax-monoidal} 
The graded Thom spectrum functor $T$ in \eqref{eq:V-space-Thom}
is lax symmetric monoidal.
\end{proposition}
\begin{proof} Since the functor $\rho$ is symmetric monoidal and $V \to \rho^*\Gr$ is a map of commutative $\cW$-space monoids, it follows that the first functor $p^*$ in the definition of $T$ is lax symmetric monoidal. The second functor $\mS^{\cW}$ is strong symmetric monoidal by Proposition~\ref{prop:S-Omega-adjunction}.
\end{proof}

Our construction of the graded Thom spectrum functor is homotopy invariant in the appropriate sense.

\begin{proposition}\label{prop:V-homotopy-invariance}
The graded Thom spectrum functor in \eqref{eq:V-space-Thom} takes $\cV$-equivalences over $\Gr$ to stable equivalences of orthogonal spectra.
\end{proposition}
\begin{proof}
We first show that $T$ takes level equivalences over $\Gr$ to level equivalences of orthogonal spectra. 
 For this we use the explicit description of $\mS^{\cW}$ from Proposition~\ref{prop:graded-SW-formula} and consider the pullback diagrams
\[
\xymatrix@-1pc{
p^*_f(X)(n_1,n_2)\times_{O(n_2)} S^{n_2} \ar[d]\ar[r]& V_{n_2}(\oplus^{\infty}\mR^{n_1})\times_{O(n_2)} S^{n_2}\ar[d] \\
(\rho^* X)(n_1,n_2)\ar[r]& (\rho^* \Gr)(n_1,n_2).
}
\]
These diagrams are homotopy cartesian since the vertical map on the right is a Hurewicz fibration and the argument given in the proof of~\cite[Lemma 2.3]{Schlichtkrull-Thom_symmetric} shows that the canonical section of the vertical map on the left is a Hurewicz cofibration. This gives the result about level equivalences. 

Now let $\Gamma$ be a level-wise Hurewicz fibrant replacement functor defined as in \cite{Schlichtkrull-Thom_symmetric}*{Section~5}. It follows from the first part of the proof that there is a natural level equivalence $T\to T\circ \Gamma$. Using this, the proposition can be proven by arguments analogous to (but slightly simpler than) those used in \cite{Schlichtkrull-Thom_symmetric}*{Section~5} to prove the corresponding statement for (ungraded) symmetric spectra. The analogue of the lifting functor in \cite{Schlichtkrull-Thom_symmetric}*{Section~4.2} is the $\cV$-space lifting functor introduced in 
Section~\ref{sec:lifting-space-level} below.
\end{proof}

Notice in particular that by the pullback square~\eqref{eq:grassmannians-stiefel-pullback}, applying the graded Thom spectrum functor to the canonical map $\tilde\Gr_{\even}\to\Gr$, we get the evenly graded oriented real cobordism spectrum $\MSOPeven$

\subsection{Lifting of space level data}\label{sec:lifting-space-level} We also want to construct graded Thom spectra from space level data and for this it will be convenient to use the homotopy colimit $\Gr_{h\cV}$ as a model of the classifying space $BO\times \mZ$, cf.\ Lemma~\ref{lem:Gr_hV_identification} and Section~\ref{subsec:Grothendieck-construction}. The point is that since the underlying $\cV$-space of $\Gr$ is not cofibrant, we cannot apply the colimit functor to $\Gr$ directly. For this reason we instead consider the bar resolution $\overline\Gr$ and the pair of Quillen adjunctions

\begin{equation}\label{eq:TopV/Gr-adjunctions} \xymatrix{ \Top/\Gr_{h\cV} \ar@<-0.5ex>[r] & \Top^{\cV}/\overline\Gr \ar@<-0.5ex>[l] \ar@<0.5ex>[r] &\Top^{\cV}/\Gr \ar@<0.5ex>[l] }
\end{equation} in which the adjunction on the left is induced by the usual $\colim$/constant functor adjunction, using the identification $\colim_{\cV}\overline\Gr =\Gr_{h\cV}$, and the adjunction on the right is induced by the canonical level equivalence $\epsilon\colon \overline\Gr\to\Gr$.

\begin{proposition}\label{prop:TopV/Gr-adjunctions} The Quillen adjunctions in \eqref{eq:TopV/Gr-adjunctions} are Quillen equivalences. \qed
\end{proposition}

Now we can define the graded Thom spectrum functor on $\Top/\Gr_{h\cV}$ as the composition
\begin{equation}\label{eq:Top-Thom-functor} T\colon \Top/\Gr_{h\cV}\xrightarrow{\Gamma} \Top/\Gr_{h\cV}\to \Top^{\cV}/\Gr\xr{p^*} \Top^{\cW}/V\xr{\mS^{\cW}} \SpO/\MOP
\end{equation} 
where $\Gamma$ is the standard Hurewicz fibrant replacement functor arising from the path space fibration (see e.g.\ \cite[Section 2]{Schlichtkrull-Thom_symmetric}) and the second arrow is the composition in \eqref{eq:TopV/Gr-adjunctions}.  It will always be clear from the context whether we view $T$ as a functor on $\Top^{\cV}/\Gr$ or on $\Top/\Gr_{h\cV}$. The use of $\Gamma$ is inspired by the work of Lewis \cite[Section IX]{LMS} and ensures  that the composition of the first two functors
  in~\eqref{eq:Top-Thom-functor} sends weak homotopy equivalences over
  $\Gr_{h\cV}$ to level equivalences of $\cV$-spaces. Hence the first part of the proof of Proposition~\ref{prop:V-homotopy-invariance} gives the following result. 
\begin{proposition}
The Thom spectrum functor $T$ in \eqref{eq:Top-Thom-functor} sends weak homotopy equivalences over $\Gr_{h\cV}$ to level equivalences of orthogonal spectra.  \qed
\end{proposition}

The construction of the graded Thom spectrum functor presented here has good multiplicative properties: Let $\cD$ be an operad in $\Top$ augmented over the Barratt--Eccles operad $\cE$ and write $\Top[\cD]$ for the category of $\cD$-algebras in $\Top$ and $\SpO[\cD]$ for the category of $\cD$-algebras in $\SpO$. The point of having $\cD$ augmented over $\cE$ is that $\Gr_{h\cV}$ then inherits the structure of a $\cD$-algebra, cf.\ \cite{Schlichtkrull-Thom_symmetric}*{Section~6}. The following theorem can be proved by a slight modification of the argument used in the proof of \cite[Corollary~6.9]{Schlichtkrull-Thom_symmetric}.

\begin{theorem}\label{thm:D-algebra-Thom-functor} Let $\cD$ be an operad augmented over the Barratt--Eccles operad $\cE$. Then the graded Thom spectrum functor in \eqref{eq:Top-Thom-functor} induces a functor of $\cD$-algebras
\[ T\colon \Top[\cD]/\Gr_{h\cV}\to \SpO[\cD]/MOP. \eqno\qed
\]
\end{theorem}

\subsection{Graded Thom spectra from continuous functors}\label{subsec:Thom-from-continous}
Let $\cK$ be a small topological category and let $F \colon \cK \to \cW$ be a continuous functor. We assume that the morphism spaces in $\cK$ are cofibrant and that the inclusions of the identity morphisms are cofibrations. The Stiefel
$\cW$-space $V$ pulls back to a $\cK$-space $F^*(V)$ and we shall write
$V_{F} = F_*^h(F^*(V))$ for the homotopy left Kan extension of
$F^*(V)$ along~$F$. Notice that the $\cW$-space $V_F$ is cofibrant by Lemma~\ref{lem:homotopy-Kan-cofibrant} and the conditions on $\cK$. It is plausible that $\mS^{\cW}[V_F]$ should be a model for the graded Thom spectrum of the map 
$BF\colon B\cK \to B\cW$ via the equivalence $B\cW\simeq \Gr_{h\cV}$ from Lemma~\ref{lem:Gr_hV_identification}. 
In order to make this precise, we introduce the map of spaces
\[
i_F\colon\rho_*(V_{F})_{h\cV} \to \rho_*(V)_{h\cV}=\Gr_{h\cV}
\]
induced by the canonical map of $\cW$-spaces $V_F\to V$. Consider the diagram of spaces
\[
\xymatrix@-1pc{
\rho_*(V_{F})_{h\cV}  \ar[d]_{i_F} & \rho^h_*(V_{F})_{h\cV}\ar[d]\ar[l] \ar[r] & (V_{F})_{h\cW}\ar[d]  \ar[r] & F^*(V)_{h\cK}\ar[d]  \ar[r]  & B\cK\ar[d]^{BF} \\
\rho_*(V)_{h\cV} & \rho^h_*(V)_{h\cV} \ar[l] \ar[r] & V_{h\cW} \ar@{=}[r]  & V_{h\cW} \ar[r] & B\cW
}
\]
where the maps are defined as in Lemma~\ref{lem:Gr_hV_identification}. The squares are all commutative except for the square with the identity on $V_{h\cW}$ at the bottom. It follows from \cite[Lemma~A.3]{Schlichtkrull-Thom_symmetric} that the latter square is homotopy commutative with a canonical choice of homotopy. Furthermore, the arguments in the proof of 
Lemma~\ref{lem:Gr_hV_identification} show that the horizontal maps are all weak homotopy equivalences except possibly for the map $F^*(V)_{h\cK}\to B\cK$. The latter is a weak homotopy equivalence if we assume in addition that the full subcategory of $\cK$ generated by the objects mapped by $F$ to objects $(n_1,n_2)$ with $n_1\neq 0$ or  $n_2=0$ is homotopy cofinal. Thus, under this assumption, we may view $T(i_F)$ as the Thom spectrum associated to $BF\colon B\cK\to B\cW$.

\begin{proposition}
There is a canonical stable equivalence $\mS^{\cW}[V_F]\simeq T(i_F)$. 
\end{proposition}
\begin{proof}
On the one hand we can view $V_F$ as the pullback of $V$ along the map $\rho^*\rho_*V_F \to \rho^*\Gr$
and on the other hand we have the equivalence $\overline{\rho_*V_F}\to \rho_*V_F$ and the homotopy cartesian square
\[
\xymatrix@-1pc{
\overline{\rho_*V_F} \ar[r] \ar[d]& \overline \Gr\ar[d] \\
\const_{\cV}(\rho_*V_F)_{h\cV} \ar[r]^-{i_F} & \const_{\cV}\Gr_{h\cV}
}
\]
in $\Top^{\cV}$. By the homotopy invariance in Proposition~\ref{prop:V-homotopy-invariance}, this gives the result.
\end{proof}
Now suppose that $\cK$ has the structure of a permutative category and that $F \colon \cK \to \cW$ is strict symmetric monoidal. Then it follows from the discussion in \cite{Schlichtkrull-Thom_symmetric}*{Section~6} that $\rho_*(V_{F})_{h\cV}$ is an $E_{\infty}$ space and that $i_F$ is an $E_{\infty}$ map (in fact, a map of $\cE$-algebras). 
Hence $T(i_F)\to  \MOP$ is a map of $E_{\infty}$ orthogonal ring spectra by Theorem \ref{thm:D-algebra-Thom-functor}.
 
\begin{example}\label{ex:MUP}
  The collection of unitary groups $U(n)$ gives rise to a permutative
  topological category $\cU$ under direct sum. Applying Quillen's localization
  construction, we obtain a permutative topological category 
 $ \cU^{-1}\cU$ with $B( \cU^{-1}\cU)\simeq BU\times \mZ$. The
  canonical maps $U(n) \to O(2n)$ induce a strict symmetric monoidal functor
  $\cU^{-1}\cU\to \cW$ and the above construction provides an
  $E_{\infty}$ orthogonal ring spectrum 
  $\MUP=T(i_{\cU^{-1}\cU})$ that models the $2$-periodic
  complex cobordism spectrum.
\end{example}

\section{Orientations and the graded Thom isomorphism}\label{sec:graded-Thom-iso} In this section we discuss orientations and the resulting graded Thom isomorphisms. We first do this directly on the level of $\cW$-spaces and then derive the corresponding results for spaces over $\Gr_{h\cV}$.
\subsection{Orientations and the graded Thom isomorphism for $\cW$-spaces} In the discussion of orientations and the graded Thom isomorphism for $\cW$-spaces, it will be useful to introduce a graded version of the usual tensor structure on $\SpO$ (see \cite[Section 1]{MMSS}). By a graded space $K=\{K_{\{d\}}\colon d\in \mZ\}$ we understand a sequence of (unbased) spaces $K_{\{d\}}$ indexed by $d\in\mZ$. Given a graded orthogonal spectrum $E$ and a graded space $K$, we write $E\bigtriangleup K$ for the graded orthogonal spectrum with $d$th term
\[ (E\bigtriangleup K)_{\{d\}}=E_{\{d\}}\wedge (K_{\{d\}+}).
\] For a $\cW$-space $X$ we shall view the corresponding colimit $X_{\cW}$ as a graded space with components $X_{\cW_{\{d\}}}$ defined by evaluating the colimits over the subcategories $\cW_{\{d\}}$.  With this notation, every $\cW$-space $X$ gives rise to a ``diagonal'' map of graded orthogonal spectra
\begin{equation}\label{eq:SW-diagonal} \delta\colon \mS^{\cW}[X]\to \mS^{\cW}[X]\bigtriangleup X_{\cW}
\end{equation} with $d$th component
\[ \mS^{\cW}[X_{\{d\}}]\to \mS^{\cW}[X_{\{d\}}\times X_{\cW_{\{d\}}}]\xr{\cong} \mS^{\cW}[X_{\{d\}}]\wedge (X_{\cW_{\{d\}}+})
\] induced from the identity $X_{\{d\}}\to X_{\{d\}}$ and the projection $X_{\{d\}}\to X_{\cW_{\{d\}}}$. Now let us fix a commutative $\cW$-space monoid $W$ and consider the over-category $\Top^{\cW}/W$ of $\cW$-spaces over $W$. Each object $X\to W$ gives rise to a map of graded orthogonal spectra
\[ \mS^{\cW}[X]\xr{\delta} \mS^{\cW}[X]\bigtriangleup X_{\cW}\to \mS^{\cW}[W]\bigtriangleup X_{\cW}
\] which in turn admits a canonical extension to a map of graded $\mS^{\cW}[W_{\{0\}}]$-modules
\begin{equation}\label{eq:W_thom-diagonal} \delta_W\colon \mS^{\cW}[W_{\{0\}}]\wedge \mS^{\cW}[X]\to \mS^{\cW}[W]\bigtriangleup X_{\cW}.
\end{equation} Notice, that this is a natural transformation when we view the domain and codomain as functors on $\Top^{\cW}/W$, and that the latter over-category inherits the structure of a symmetric monoidal category from $\Top^{\cW}$ since we assume $W$ to be commutative. The fact that the functor $\mS^{\cW}$ is strong symmetric monoidal implies that the domain in \eqref{eq:W_thom-diagonal} is strong symmetric monoidal as a functor from $\Top^{\cW}/W$ to the category of graded $\mS^{\cW}[W_{\{0\}}]$-modules. Furthermore, we may view the colimit functor on $\Top^{\cW}$ as a strong symmetric monoidal functor, $X_{\cW}\times Y_{\cW}\cong (X\boxtimes Y)_{\cW}$, so that the codomain inherits the structure of a lax symmetric monoidal functor to graded $\mS^{\cW}[W_{\{0\}}]$-modules (but notice that $\mS^{\cW}[W]\bigtriangleup X_{\cW}$ is not an $\mS^{\cW}[W]$-module).

\begin{lemma}\label{lem:W_thom-diagonal} The map $\delta_W$ of graded $\mS^{\cW}[W_{\{0\}}]$-modules defines a monoidal natural transformation between lax symmetric monoidal functors. It is a stable equivalence provided that $W$ is grouplike and that $X$ is cofibrant.
\end{lemma}
\begin{proof} For the statement about monoidality of $\delta_W$ we remark that the map $\delta$ in \eqref{eq:SW-diagonal} may be viewed as a monoidal natural transformation between lax symmetric monoidal functors with values in $\Grad_{\mZ}\SpO$. From this it follows that $\delta_W$ is a composition of monoidal natural transformations, hence itself monoidal.

Now suppose that $W$ is grouplike. In order to show that $\delta_W$ is a stable equivalence for cofibrant $X$, we first consider the case of a free $\cW$-space of the form $F^{\cW}_{(n_1,n_2)}(*)$. The latter is concentrated in degree $n=n_2-n_1$ and the colimit over $\cW_{\{n\}}$ is a one-point space. 
We observe that in this case $\delta_W$ can be identified with the stable equivalence obtained by setting $R=\mS^{\cW}[W_{\{0\}}]$ in Proposition~\ref{prop:SW-M-equivalence}. It then follows by homotopy invariance that the lemma holds for $\cW$-spaces of the form $F_{(n_1,n_2)}^{\cW}(D^n)$ and by induction on $n$ that it holds for $\cW$-spaces of the form $F_{(n_1,n_2)}^{\cW}(S^n)$. Here we have used that the stable model structure on $\SpO$ satisfies the gluing lemma \cite[Theorem~7.4]{MMSS} and the monoid axiom \cite[Proposition~12.5]{MMSS} in order to avoid making additional cofibrancy conditions on $W$. It now follows by induction that the lemma holds for all cell complexes constructed from the generating cofibrations which gives the result since every cofibrant $\cW$-space is a retract of such a cell complex.
\end{proof}

We now set up a more general framework that allows us to establish a Thom isomorphism theorem for $\cW$-spaces that are not necessarily cofibrant and augmented over $W$. 

\begin{definition}\label{def:W-orientation} Given a grouplike and positive fibrant commutative $\cW$-space monoid $W$, a \emph{$W$-orientation} of a $\cW$-space $X$ is a diagram of the form
\[ X\xl{\sim} X' \to W
\] where $X'$ is a $\cW$-space and the first map is a $\cW$-equivalence as indicated. We say that a $\cW$-space is \emph{oriented} (resp.\ \emph{orientable}) if it is (resp.\ can be) equipped with a $W$-orientation. A map of $W$-oriented $\cW$-spaces is defined in the obvious way as a pair of compatible maps relating the orientations.
\end{definition}

This notion is designed so as to be homotopy invariant: The positive fibrancy condition on $W$ ensures that if $X$ is $W$-orientable, then there exists a $W$-orientation in which $X'\to X$ is a positive acyclic fibration. Since positive acyclic fibrations are preserved under pullback, this in turn implies that if $Y\to X$ is a map of \mbox{$\cW$-spaces} and $X$ is $W$-orientable, then also $Y$ is $W$-orientable. Furthermore, if $X$ is $W$-orientable and $X\to Y$ is a $\cW$-equivalence, then clearly $Y$ is also $W$-orientable.

Notice, that if a $\cW$-space $X$ is $W$-orientable, then $X$ is also orientable with respect to a cofibrant replacement 
$W^{\cof}\xr{\sim} W$ (an acyclic fibration) in the positive model structure on commutative $\cW$-space monoids established in Theorem~\ref{thm:positive-model-structure}. For the Thom isomorphism it will be important that the orthogonal spectrum $\mS^{\cW}[W]$ is homotopically well-behaved and for this reason we adopt the following standing assumptions for the rest of this subsection:
\textbf{The grouplike commutative $\cW$-space monoid $W$ is assumed to be cofibrant as well as fibrant in the positive model structure on commutative $\cW$-space monoids.} This is not a serious restriction since any commutative $\cW$-space monoid admits a cofibrant fibrant replacement.

Let $X\xl{\sim} X' \to W$ be an oriented $\cW$-space and let as usual $\overline{X'}$ denote the bar resolution on $X'$. The given maps $\overline{X'}\to X'\to W$ then define an object in $\Top^{\cW}/W$ and we may form the composition
\[ \mS^{\cW}[W_{\{0\}}]\wedge \mS^{\cW}[\overline{X'}]\xr{\delta_W} \mS^{\cW}[W]\bigtriangleup X'_{h\cW} \to \mS^{\cW}[W]\bigtriangleup X_{h\cW}
\] where the first map is the natural transformation \eqref{eq:W_thom-diagonal} applied to $\overline{X'}$ and the second map is induced by the map of homotopy colimits $X'_{h\cW}\to X_{h\cW}$. Here we identify the colimit $\overline{X'}_{\cW}$ with the homotopy colimit $X'_{h\cW}$. Using the given maps $\overline{X'}\to\overline X\to X$, we thus get a chain of maps of graded $\mS^{\cW}[W_{\{0\}}]$-modules
\begin{equation}\label{eq:W-hocolim-thom-diagonal} \mS^{\cW}[W_{\{0\}}]\wedge \mS^{\cW}[X]\leftarrow \mS^{\cW}[W_{\{0\}}]\wedge \mS^{\cW}[\overline {X'}]\to \mS^{\cW}[W]\bigtriangleup X_{h\cW}
\end{equation} that are natural with respect to maps of $W$-oriented $\cW$-spaces.

\begin{lemma}\label{lem:W-hocolim-thom-diagonal} If the $W$-oriented $\cW$-space $X$ is $\mS^{\cW}$-good, then the natural maps in \eqref{eq:W-hocolim-thom-diagonal} are stable equivalences. 
\end{lemma}
\begin{proof} Using that $\overline {X'}$ is always $\mS^{\cW}$-good by Proposition~\ref{prop:bar-SW-good},
we may assume without loss of generality that $X'$ is cofibrant. Then $\overline{X'}$ is also cofibrant, so that the last map in \eqref{eq:W-hocolim-thom-diagonal} is a stable equivalence by Lemma~\ref{lem:W_thom-diagonal}.
We also know that the induced map $\mS^{\cW}[\overline{X'}] \to \mS^{\cW}[X]$ is a stable equivalence because $X$ and $\overline{X'}$ are $\mS^{\cW}$-good. Since the spectrum $\mS^{\cW}[W_{\{0\}}]$ is a retract of $\mS^{\cW}[W]$, it follows from \cite[Proposition 1.3.11 and Theorem 1.3.29]{Stolz_equivariant} and the cofibrancy assumption on $\mS^{\cW}[W]$ that $\mS^{\cW}[W_{\{0\}}]\sm -$ preserves stable equivalences. 
\end{proof} 
In the following we shall fix a map of commutative orthogonal ring spectra $\mS^{\cW}[W_{\{0\}}]\to R$ making $R$ a cofibrant $\mS^{\cW}[W_{\{0\}}]$-algebra and consider the commutative graded orthogonal ring spectrum
\[ 
RP=\mS^{\cW}[W]\wedge _{\mS^{\cW}[W_{\{0\}}]} R.
\] 
introduced in Definition~\ref{def:RP}.  Applying the functor 
$-\wedge _{\mS^{\cW}[W_{\{0\}}]} R$ to the chain of stable equivalences in Lemma~\ref{lem:W-hocolim-thom-diagonal}, we get the graded ``Thom isomorphism'' in the following theorem.

\begin{theorem}\label{thm:R-Thom-iso}
Let $R$ be a cofibrant $\mS^{\cW}[W_{\{0\}}]$-algebra and let $X$ be an $\mS^{\cW}$-good $\cW$-space equipped with a $W$-orientation. Then there is a chain of natural stable equivalences
$
R\wedge \mS^{\cW}[X]\simeq RP\bigtriangleup X_{h\cW}
$
of graded $R$-modules. \qed
\end{theorem}

\subsection{Multiplicative orientations}\label{subsec:mult-orientations}
Let $\cD$ be an operad of topological spaces in the sense of \cite{May_geometry}, and let $\mD$ be the corresponding monad on $\Top^{\cW}$ defined by 
\[
\mD(X)=\coprod_{n\geq 0}\cD(n)\times _{\Sigma_n}X^{\boxtimes n}, 
\]
where $X^{\boxtimes 0}$ denotes the 
monoidal unit $U^{\cW}$. By definition, a \emph{$\cD$-algebra in $\Top^{\cW}$} is an algebra for the monad $\mD$, and we write $\Top^{\cW}[\cD]$ for the category of $\cD$-algebras. It follows from the universal property of the $\boxtimes$-product that a $\cD$-algebra structure on a $\cW$-space $X$ amounts to a sequence of natural transformations 
\[
\Theta_k\colon \cD(k)\times X(n_1^1,n_2^1)\times \dots\times X(n_1^k,n_2^k)
\to X((n_1^1,n_2^1)\oplus\dots\oplus(n_1^k,n_2^k))
\]
of functors $\cW^k\to \Top$, subject to the usual associativity, unitality, and equivariance conditions, cf.\ \cite[Lemma~1.4]{May_geometry}. The canonical projection of $\cD$ onto the terminal ``commutativity'' operad allows us to view any commutative $\cW$-space monoid as a $\cD$-algebra. 

\begin{definition}\label{def:multiplicative-orientation}
Given a grouplike and positive fibrant commutative $\cW$-space monoid $W$, a \emph{$\cD$-multiplicative $W$-orientation} of a $\cD$-algebra $X$ in $\Top^{\cW}$ is a specified lift of a $W$-orientation $X\xl{\sim} X'\to W$ to a diagram in $\Top^{\cW}[\cD]$.
\end{definition}

In connection with $\cW$-spaces it is again convenient to consider operads augmented over the Barratt--Eccles operad $\cE$. The reason for this comes from the following lemma which is proved by an argument similar to that used in the proof of the analogous result for $\cI$-spaces \cite[Lemma~6.7]{Schlichtkrull-Thom_symmetric}. 

\begin{lemma}\label{lem:Barratt-Eccles-bar-resolution}
Let $\cD$ be an operad augment over $\cE$. Then the bar resolution $X\mapsto \overline X$ induces a functor $\Top^{\cW}[\cD]\to\Top^{\cW}[\cD]$ such that the natural level equivalence $\overline X\to X$ is a map of $\cD$-algebras. \qed
\end{lemma}

Since the colimit functor on $\Top^{\cW}$ is strong symmetric monoidal and we have the identification $\overline{X}_{\cW}=X_{h\cW}$, it follows from the lemma that the homotopy colimit functor  induces a functor $\Top^{\cW}[\cD]\to \Top[\cD]$, taking $X$ to $X_{h\cW}$. 

\begin{example}
With notation as in Sections~\ref{sec:oriented-graded-grass} and~\ref{subsec:Thom-from-continous}, a factorization of a continuous strict symmetric monoidal functor $F\colon \cK \to \cW$ through $\iota_{\even}\colon S\cW_{\even} \to \cW$ induces an $\cE$-multiplicative orientation of $V_{F}$ with respect to $\tilde V_{\even}$. 
\end{example}

We wish to have a multiplicative version of the graded Thom isomorphism in 
Theorem~\ref{thm:R-Thom-iso}. An operad $\cD$ as above induces a monad $\mD$ on the category $\Grad_{\mZ}\SpO$ in the usual way,
\[
\mD(E)=\bigvee_{n\geq 0}\cD(n)_+\wedge _{\Sigma_n}E^{\wedge n}, 
\]
and we define a \emph{graded $\cD$-algebra} to be an algebra for this monad.
Given a commutative orthogonal ring spectrum $R$, we shall use the term \emph{a graded $\cD$-algebra under $R$} to mean a graded orthogonal spectrum $A$ equipped with the structure of a graded $\cD$-algebra and a map of graded $\cD$-algebras $R\to A$, where we view $R$ as a graded $\cD$-algebra concentrated in degree zero.

\begin{theorem}\label{thm:mult-R-Thom-iso}
Let $R$ be a cofibrant commutative $\mS^{\cW}[W_{\{0\}}]$-algebra, let $\cD$ be an operad augmented over $\cE$, and let $X$ be an $\mS^{\cW}$-good \mbox{$\cD$-algebra} in $\Top^{\cW}$ equipped with a $\cD$-multiplicative $W$-orientation with cofibrant $W$. Then there is a chain of stable equivalences $R\wedge \mS^{\cW}[X]\simeq RP\bigtriangleup X_{h\cW}$ of graded $\cD$-algebras under~$R$. 
\end{theorem}
\begin{proof}
We must show that with the given assumptions, the chain of stable equivalences in 
Theorem~\ref{thm:R-Thom-iso} is in fact a chain of maps of $\cD$-algebras under $R$. For this we 
use that the monoidal natural transformation $\delta_W$ in Lemma~\ref{lem:W_thom-diagonal} takes every $\cD$-algebra over $W$ to a map of graded $\cD$-algebras under $\mS^{\cW}[W_{\{0\}}]$. By Lemma~\ref{lem:Barratt-Eccles-bar-resolution}, this applies in particular to the $\cD$-algebra $\overline{X'}\to X'\to W$ given by the orientation. The result now follows since the extension of scalars along $\mS^{\cW}[W_{\{0\}}]\to R$ is a symmetric monoidal functor and hence preserves operad actions.   
\end{proof}

If the operad $\cD$ in Definition~\ref{def:multiplicative-orientation} is an $E_{\infty}$-operad, then we say that $X$ has an $E_{\infty}$ $W$-orientation. In case the $E_{\infty}$ operad $\cD$ is not itself augmented over $\cE$, we may replace it by the product operad 
$\cD\times \cE$, which is then an $E_{\infty}$ operad that is augmented over $\cE$. 

\begin{corollary}\label{cor:D-multiplicative-Thom-iso} 
Suppose that the operad $\cD$ in Theorem~\ref{thm:mult-R-Thom-iso} is an $E_{\infty}$ operad (such that $X$ has an $E_{\infty}$ orientation). Then $R\wedge \mS^{\cW}[X]\simeq RP\bigtriangleup X_{h\cW}$ as graded $E_{\infty}$ $R$-algebras. \qed
\end{corollary}

The above corollary applies in particular when $M=X$ is a cofibrant commutative $\cW$-space monoid equipped with a map of commutative $\cW$-space monoids $M\to W$ in which case we may set $\cD=\cE$. In general, the corollary shows that the graded Thom isomorphism is compatible with the Dyer--Lashof operations arising from the $E_{\infty}$ structures. 

We now prepare to give an algebraic formulation of the graded Thom isomorphism.
If a graded topological space $K=\{K_n\colon n\in \mZ\}$ has the extra structure of a graded monoid with multiplication defined by maps 
$K_m\times K_n\to K_{m+n}$, then $RP\bigtriangleup K$ inherits the structure of a graded orthogonal ring spectrum. We wish to describe $\pi_*(RP\bigtriangleup K)$ solely in terms of the $R$-homology of $K$. To this end, let $d\geq 0$ denote the periodicity degree of $W$ as in Section~\ref{sec:grouplike} and let us introduce the notation
\[
R_{\circledast}(K)=\bigoplus_{n\in d\mZ}\Sigma^nR_*(K_n),
\] 
where $R_*(-)=\pi_*(R\wedge(-)_+)$ denotes the unreduced homology theory associated with $R$. We think of $R_{\circledast}(K)$ as a bigraded $\pi_*(R)$-module with $(m,n)$th term $\Sigma^nR_m(K_n)$ (that is, $R_{m-n}(K_n)$), and we write $\sigma^n\alpha$ for the element $\alpha\in R_{m-n}(K_n)$ when thought of as an element in bidegree $(m,n)$. When $K$ is a graded monoid we give $R_{\circledast}(K)$ the structure of a bigraded $\pi_*(R)$-algebra with multiplication
\begin{equation}\label{eq:circledast-multilplication}
\sigma^m\alpha\cdot\sigma^n\beta=\sigma^{m+n}(\alpha\cdot\beta),
\end{equation} 
where $\alpha\cdot\beta$ denotes the product in the homotopy groups of the graded orthogonal ring spectrum $c(R)\bigtriangleup K$. If $K$ is homotopy commutative, then this multiplication is graded commutative in the sense that $\sigma^m\alpha\cdot\sigma^n\beta=(-1)^{|\alpha||\beta|}\sigma^n\beta\cdot\sigma^m\alpha$.

\begin{remark}
It may seem natural to introduce the sign $(-1)^{|\alpha|n}$ in the multiplication 
\eqref{eq:circledast-multilplication}. However, this does not change the product since if $d$ is odd, then $\pi_*(R)$ is necessarily an $\mF_2$-algebra.
\end{remark}

\begin{proposition}
For a graded topological space $K$, there is an isomorphism of bigraded $\pi_*(R)$-modules
$
R_{\circledast}(K)\cong \pi_*(RP\bigtriangleup K)
$.
If $K$ is a graded monoid, then this is an isomorphism of bigraded 
$\pi_*(R)$-algebras. 
\end{proposition}
\begin{proof}
With $d\geq 0$ the periodicity degree of $W$, we choose an element $u\in W(d_1,d_2)$ such that $d=d_2-d_1$. Since $[u]\in \pi_d(RP_{\{d\}})$ is a unit in $\pi_*(RP)$, we have elements $[u]^{n/d}\in 
\pi_n(RP_{\{n\}})$ for all $n\in d\mZ$ with $[u]^0$ being the multiplicative unit. The isomorphism in the proposition is then defined by the composite homomorphisms
\[
\Sigma^n\pi_*(R\wedge K_{n+})\xr{[u]^{n/d}} \pi_*(RP_{\{n\}}\wedge R\wedge K_{n+})\to \pi_*(RP_{\{n\}}\wedge K_{n+}),
\] 
where the first homomorphism is exterior multiplication by $[u]^{n/d}$ and the second is induced by the multiplication in $RP$. These isomorphisms only depend on the choice of $[u]$, hence are compatible and define a multiplicative isomorphism when $K$ is a graded monoid.
\end{proof}

\begin{corollary}
With $R$ and $X$ as in Corollary~\ref{cor:D-multiplicative-Thom-iso}, there is an isomorphism of graded $\pi_*(R)$-algebras $R_*(\mS^{\cW}[X])\cong R_{\circledast}(X_{h\cW})$. \qed 
\end{corollary}

\subsection{Orientability with respect to the Stiefel $\cW$-spaces $V$ and $\tilde V_{\even}$}
\label{subsec:orientability-stiefel-cW-sp} 
Now we specialize to orientations with respect to the Stiefel $\cW$-spaces $V$ and $\tilde V_{\even}$. These are of course closely related to orientations with respect to the Eilenberg--Mac Lane spectra $H\mZ/2$ and 
$H\mZ$. In the case of $V$, we first choose a cofibrant replacement $V^{\cof}$ and then we realize the 0th Postnikov section $\mS^{\cW}[V_{\{0\}}^{\cof}]\to H\mZ/2$ as a map of commutative orthogonal ring spectra (see e.g.\ \cite{Schwede-SymSp}).  Using the positive stable model structure on commutative $\mS^{\cW}[V_{\{0\}}^{\cof}]$-algebras, we may further assume that $H\mZ/2$ is  positive fibrant  and cofibrant as a commutative $\mS^{\cW}[V_{\{0\}}^{\cof}]$-algebra. Passing to graded units, we get a $\cW$-equivalence $V_{\{0\}}^{\cof}\to \GL_1^{\cW}(H\mZ/2)$. 
By Lemma~\ref{lem:VhW-BW-equivalence}, every \mbox{$\cW$-space} is orientable with respect to $V$ and hence with respect to $V^{\cof}$. 
This gives the graded Thom isomorphism
\[
H_*(\mS^{\cW}[X],\mZ/2)\cong H\mZ/2P_{\circledast}(X_{h\cW})= \bigoplus_{n\in \mZ}
\Sigma^nH_*(X_{h\cW_{\{n\}}},\mZ/2)
\]
for every $\mS^{\cW}$-good $\cW$-space $X$. 
Here $H\mZ/2P$ is the graded Eilenberg--Mac Lane spectrum 
$\mS^{\cW}[V^{\cof}]\wedge_{\mS^{\cW}[V_{\{0\}}^{\cof}]} H\mZ/2$ with 
$\pi_*(H\mZ/2P)=\mZ/2[t,t^{-1}]$ where $t$ is a generator of degree $|t|=1$. 

In the oriented case of $\tilde V_{\even}$, we choose a cofibrant replacement $\tilde V_{\even}^{\cof}$ and realize the 0th Postnikov section $\mS^{\cW}[\tilde V_{\{0\}}^{\cof}]\to H\mZ$ as a map of commutative orthogonal ring spectra making $H\mZ$ a positive fibrant and cofibrant $\mS^{\cW}[\tilde V_{\{0\}}^{\cof}]$-algebra. Passing to graded units, we now get a \mbox{$\cW$-equivalence} $\tilde V_{\{0\}}^{\cof}\to\GL^{\cW}_1(H\mZ)$. This shows that the condition for a \mbox{$\cW$-space} to be orientable with respect to $\tilde V_{\even}$  is compatible with the usual notion of orientability. In this case we get an integral graded Thom isomorphism
\[
H_*(\mS^{\cW}[X],\mZ)\cong H\mZ P_{\even\circledast}(X_{h\cW})=\bigoplus_{n\in \mZ}\Sigma^{2n}
H(X_{h\cW_{\{2n\}}},\mZ) 
\]
for every $\mS^{\cW}$-good $\cW$-space $X$ equipped with a $\tilde V_{\even}$-orientation.
Here $H\mZ P_{\even}$ denotes the evenly graded Eilenberg--Mac Lane spectrum $\mS^{\cW}[\tilde V_{\even}^{\cof}]\wedge_{\mS^{\cW}[\tilde V_{\{0\}}^{\cof}]}H\mZ$ with $\pi_*(H\mZ P_{\even})=\mZ[t,t^{-1}]$ where $t$ is a generator of degree $|t|=2$. 

\subsection{Orientability with respect to $\GL_1^{\cW}(R)$}\label{subsec:GL_1W(R)-orientations}
 Let $R$ be a positive fibrant commutative orthogonal ring spectrum and consider the positive fibrant commutative $\cW$-space monoid $\GL_1^{\cW}(R)$. As we explain now, orientations with respect to the sub commutative $\cW$-space monoid $\GL_1^{\cW}(R)_{\{0\}}$ are equivalent to orientations in the original (ungraded) sense of May, Quinn, Ray, and Tornehave~\cite{May-ring-spaces}. By definition, 
$\GL_1^{\cW}(R)_{\{0\}}$ is concentrated on the subcategory $\cW_{\{0\}}$ and equals $\GL_1^{\cW}(R)$ if and only if all the units in $\pi_*(R)$ have degree zero. In this way, orientations with respect to $\GL_1^{\cW}(R)$ extend the classical notion of orientations to the graded setting. 

Let $\Delta \colon \cV \to \cW$ be the canonical diagonal functor and let us write $\GL_1^{\cV}(R)$ for the pullback of $\GL_1^{\cW}(R)$ to a commutative $\cV$-space monoid.  The homotopy colimit $\GL_1^{\cV}(R)_{h\cV}$ is a model of the usual (ungraded) units of $R$ often denoted $\GL_1(R)$. (It follows from work of Lind~\cite{Lind-diagram} that $\GL_1^{\cV}(R)$ is equivalent to the analogous construction for symmetric spectra denoted $\GL_1(R)$ in \cite{Schlichtkrull-units} and also to the earlier notion of units introduced in \cite{May-ring-spaces}.)
 In order to compare the notions of units represented by $\GL_1^{\cV}(R)$ and $\GL_1^{\cW}(R)_{\{0\}}$, we consider the functor $O\colon \cV\to\Top$ from Section~\ref{subsec:Grothendieck-construction} and equip it with the structure of a commutative \mbox{$\cV$-space} monoid under direct sum. Composing the canonical map $O\to \GL_1^{\cV}(\mS)$ with the map induced by the multiplicative unit $\mS\to R$, we get a map of commutative \mbox{$\cV$-space} monoids $O\to\GL_1^{\cV}(R)$. The map of homotopy colimits $O_{h\cV}\to\GL_1^{\cV}(R)_{h\cV}$ is therefore a map of grouplike topological monoids (in fact $E_{\infty}$ spaces). In the following we shall implicitly assume that the multiplicative unit makes $R_0$ a well-based topological monoid so that also $\GL_1^{\cV}(R)_{h\cV}$ is well-based. The following proposition gives an explicit realization of the homotopy fiber sequence \eqref{eq:BO-BGL(R)-sequence} from the introduction. 

\begin{proposition}
The homotopy fiber of the map $B(O_{h\cV})\to B(\GL_1^{\cV}(R)_{h\cV})$ is weakly equivalent to $\GL_1^{\cW}(R)_{h\cW_{\{0\}}}$.
\end{proposition}
\begin{proof}
First we observe that there is a homotopy cartesian pullback diagram 
\[
\xymatrix@-1pc{
B(\GL_1^{\cV}(R)_{h\cV},O_{h\cV},*) \ar[r] \ar[d]& B(*,O_{h\cV},*)\ar[d] \\
B(\GL_1^{\cV}(R)_{h\cV},\GL_1^{\cV}(R)_{h\cV},*) \ar[r] & B(*,\GL_1^{\cV}(R)_{h\cV},*)
}
\]
where the horizontal maps are quasifibrations by \cite{May-classifying}*{Theorem~7.6}. It follows that the homotopy fiber of the map in the proposition is weakly equivalent to $B(\GL_1^{\cV}(R)_{h\cV},O_{h\cV},*)$. Secondly, we use that 
$\GL_1^{\cW}(R)_{h\cW_{\{0\}}}$ is weakly equivalent to $\hocolim_{n\in \cV}\GL_1^{\cW}(R)(n,n)_{hO(n)}$ by Proposition~\ref{prop:Grothendieck-hocolim-equivalence}. Recall that $\GL_1^{\cW}(R)(n,n)$ denotes the union of the path components in $\Omega^n(R_n)$ that represent units in $\pi_0(R)$; for the sake of the proof we introduce the shorter notation $\Omega^n(R_n)^*$ for this subspace. The homotopy orbit space $\GL_1^{\cW}(R)(n,n)_{hO(n)}$ is with respect to the left $O(n)$-action in the second variable and it is not difficult to see that this is homeomorphic to the bar construction $B(\Omega^n(R_n)^*,O(n),*)$ where now $O(n)$ acts from the right on  $\Omega^n(R_n)^*$ by precomposition. Now define a weak homotopy equivalence
\[
B(\GL_1^{\cV}(R)_{h\cV},O_{h\cV},B\cV) \xr{\sim}
\hocolim_{n\in\cV}B(\Omega^n(R_n)^*,O(n),*)
\] 
by passing to the geometric realization of the  simplicial map that in simplicial degree $k$ is induced by the obvious natural transformation of $\cV^{k+2}$-diagrams
\[
\Omega^{n_0}(R_{n_0})^*\times O(n_1)\times \dots\times O(n_k)\to \Omega^{n}(R_{n})^*\times O(n)\times \dots\times O(n)
\]
where, for an object $(n_0,\dots, n_{k+1})$ in $\cV^{k+2}$, we set $n=n_0\oplus\dots\oplus n_{k+1}$. 
Since $B\cV$ is contractible, the left hand side of the above equivalence is equivalent to 
$B(\GL_1^{\cV}(R)_{h\cV},O_{h\cV},*)$ which gives the result.
\end{proof}
It follows from the proposition that $\GL_1^{\cW}(R)_{h\cW_{\{0\}}}$ is equivalent to the classifying space for $R$-oriented stable vector bundles introduced in~\cite{May-ring-spaces}*{IV.3}. Given a $\cW$-space $X$ equipped with an orientation $X\xl{\sim}X'\to\GL_1^{\cW}(R)$, this implies that the associated ``stable vector bundle'' $X_{h\cW_{\{0\}}}\to B\cW_{\{0\}}\simeq BO$ is $R$-oriented in the sense of \cite{May-ring-spaces}. See also the discussion in Section~\ref{subsec:space-level-Thom-iso} below.

\subsection{Orientations and graded Thom isomorphisms for spaces over $\Gr_{h\cV}$}
\label{subsec:space-level-Thom-iso}
The orientation theory for $\cW$-spaces is extended to spaces over $\Gr_{h\cV}$ via the Quillen equivalences
\[
\Top/\Gr_{h\cV}\simeq \Top^{\cV}/\Gr\simeq \Top^{\cW}/V\simeq \Top^{\cW}
\]
resulting from Propositions~\ref{prop:rho-p-Quillen-equivalence} and \ref{prop:TopV/Gr-adjunctions} and the fact that the projection of $V$ onto the terminal $\cW$-space is an acyclic $\cW$-fibration by Lemma~\ref{lem:VhW-BW-equivalence}. 
Thus, for $W$ a grouplike and positive fibrant  commutative \mbox{$\cW$-space} monoid, we say that a map of spaces $f\colon K\to \Gr_{h\cV}$ is \mbox{$W$-orientable} if the corresponding $\cW$-space admits a $W$-orientation in the sense of Definition~\ref{def:W-orientation}. If we assume in addition that $W$ is cofibrant, then this can be expressed more explicitly: Using 
Proposition~\ref{prop:V-homotopy-invariance}, one can show that such a map $f\colon K\to \Gr_{h\cV}$ is \mbox{$W$-orientable} if it is homotopic to a map that factors though the map $(\rho_*W)_{h\cV} \to \Gr_{h\cV}$ 
induced by a lift $W \to V$. 
(Here $(\rho_*W)_{h\cV} \simeq W_{h\cW}$ by the argument in Section~\ref{subsec:Thom-from-continous}.) 
In the situation of Theorem~\ref{thm:R-Thom-iso} we then get a chain of stable equivalences
\[
R\wedge T(f)\simeq RP\bigtriangleup K,
\] 
where we view $K$ as a $\mZ$-graded space with $n$th term $K_{\{n\}}=f^{-1}(\Gr_{h\cW_n})$. It follows from Theorems~\ref{thm:D-algebra-Thom-functor} and \ref{thm:mult-R-Thom-iso} that this is multiplicative in the appropriate sense: If $f$ is a map of  
$\cD$-algebras for an operad augmented over the Barratt--Eccles operad, then the stable equivalences in question are maps of $\cD$-algebras. On the level of stable homotopy groups we get an isomorphism
\[
R_*(T(f))\cong R_{\circledast}(K)=\bigoplus_{n\in \mZ}\Sigma^nR_*(K_{\{n\}}).
\]
This applies in particular for $R=H\mZ$ provided that $f$ is homotopic to a map that factors through 
$\tilde \Gr_{h\cV}$.

\section{Applications to topological logarithmic structures}
\label{sec:log-structures}
The notion of a pre-log structure on a commutative symmetric ring spectrum was introduced by Rognes~\cite{Rognes-TLS} in an ungraded version and extended to a graded version in \cite{Sagave-S_diagram}, see also~\cites{Sagave_log-on-k-theory, RSS_LogTHH-I}. 
The analogous definition in the setting of orthogonal spectra reads as follows: A \emph{pre-log structure} on a commutative orthogonal ring spectrum $R$ is given by a commutative $\cW$-space monoid $M$ and a map of commutative \mbox{$\cW$-space} monoids $M\to \Omega^{\cW}(R)$. By construction, the graded units $\GL_1^{\cW}(R)$ is a sub commutative $\cW$-space monoid of $\Omega^{\cW}(R)$ and we say that $\alpha\colon M\to  \Omega^{\cW}(R)$ is a \emph{log structure} if the induced map $\alpha^{-1}(\GL_1^{\cW}(R))\to \GL_1^{\cW}(R)$ is a $\cW$-equivalence. The inclusion $GL_1^{\cW}(R)\to \Omega^{\cW}(R)$ defines the \emph{trivial log structure} on $R$. As in the setting of symmetric spectra \cite{RSS_LogTHH-I}*{Construction~4.22}, every pre-log structure $\alpha\colon M\to  \Omega^{\cW}(R)$ has an associated log structure $\alpha^a\colon M^a\to  \Omega^{\cW}(R)$, where $M^a$ is defined as the homotopy pushout of the diagram
\[
M\leftarrow \alpha^{-1}(\GL_1^{\cW}(R))\to \GL_1^{\cW}(R)
\]
of commutative $\cW$-space monoids. The notion of a pre-log structure is placed in the context of orientation theory as follows.

\begin{proposition}\label{prop:pre-log-orientation}
A pre-log structure $\alpha\colon M\to  \Omega^{\cW}(R)$ gives rise to a $\GL_1^{\cW}(R)$-orientation of $M$ if and only if the associated log structure is $\cW$-equivalent to the trivial log structure on $M$.
\end{proposition}
\begin{proof}
If $\alpha\colon M\to  \Omega^{\cW}(R)$ factors through $\GL_1^{\cW}(R)$ (and hence defines a 
$\GL_1^{\cW}(R)$-orientation), then $M=\alpha^{-1}(\GL_1^{\cW}(R))$ so the logification $M^a$ is $\cW$-equivalent to the trivial log structure. On the other hand, if $M^a$ is $\cW$-equivalent to the trivial log structure, then $M^a$ is grouplike so $\alpha^a$ and hence $\alpha$ factors through $\GL_1^{\cW}(R)$.
\end{proof}

From the perspective of pre-log structures $\alpha\colon M\to  \Omega^{\cW}(R)$, the main interest is usually in the part of $M$ that does not map to $\GL_1^{\cW}(R)$. 

In the remainder of this subsection we show how our orientation theory for \mbox{$\cW$-spaces} can be used to derive the results needed for the calculations of logarithmic topological Hochschild homology in \cite{RSS_LogTHH-II}. In the setting of symmetric spectra, the theory of pre-log structures  is formulated in terms of $\cJ$-spaces, where $\cJ$ denotes Quillen's localization construction $\Sigma^{-1}\Sigma$ on the category $\Sigma$ with objects the finite sets $\bld n=\{1,\dots,n\}$ and morphisms the bijections. The category of $\cJ$-spaces $\Top^{\cJ}$ is studied extensively in~\cite{Sagave-S_diagram}*{Section 4}, and it shown in \cite{Sagave-S_diagram}*{Proposition 4.23} that the latter category is related to the category of symmetric spectra $\Spsym$ by an adjunction $\mS^{\cJ}\colon \Top^{\cJ}\rightleftarrows \Spsym\colon \Omega^{\cJ}$ that is analogous to the  $(\mS^{\cW},\Omega^{\cW})$ adjunction from Proposition~\ref{prop:S-Omega-adjunction}. 

Let $\Psi\colon \Sigma \to \mathcal O$ be the strict symmetric monoidal functor sending a permutation to the associated isometry and let us use the same notation for the induced functor $\Psi \colon \cJ \to \cW$. We will use the notation $(\Psi_*,\Psi^*)$ for both the resulting adjunction relating $\cJ$- and $\cW$-spaces and the resulting adjunction relating symmetric and orthogonal spectra. These functors fit into a diagram of adjunctions
\begin{equation}\label{eq:orth-sym-adjunctions}
\xymatrix@-.5pc{ \Top^{\cJ} \ar@<.4ex>[rr]^{\Psi_*}
  \ar@<-.4ex>[d]^{~\Omega^{\cJ}}& &
  \Top^{\cW}\ar@<.4ex>[ll]^{\Psi^*} \ar@<-.4ex>[d]^{~\Omega^{\cW}}\\
  \Spsym \ar@<.4ex>[rr]^{\Psi_*} \ar@<-.4ex>[u]^{\mS^{\cJ}~}& &
  \SpO\ar@<.4ex>[ll]^{\Psi^*}  \ar@<-.4ex>[u]^{\mS^{\cW}~}}
\end{equation}
where the square of right adjoints commutes and the square of left adjoints commutes up to isomorphism. Since all left adjoints are strong symmetric monoidal functors, there are corresponding squares of adjunctions for the associated categories of algebras over an operad $\cD$. Moreover, all adjunctions are Quillen adjunctions with respect to the absolute and positive model structures on the respective categories. As in the case of $\mS^{\cW}$, there is a notion of 
\emph{$\mS^{\cJ}$-good $\cJ$-spaces} on which $\mS^{\cJ}$ captures the homotopy type of the left derived functor. Useful criteria for $\mS^{\cJ}$-goodness are developed in~\cite[Section~8]{RSS_LogTHH-I}.

\begin{remark} Although the lower horizontal adjunction in~\eqref{eq:orth-sym-adjunctions} is a Quillen equivalence by \cite[Theorem 10.4]{MMSS}, this is not true for the upper horizontal adjunction relating $\cJ$- and $\cW$-spaces. 
Indeed, there are compatible Quillen equivalences $\Top^{\cJ}\simeq\Top/B\cJ$ and  $\Top^{\cW}\simeq\Top/B\cW$ (compare \cite{Sagave-S_diagram}*{Theorem~4.9}) under which $(\Psi_*,\Psi^*)$ corresponds to the adjunction induced by the canonical map $B\cJ \to B\cW$, and the latter adjunction is not a Quillen equivalence. Applying $\mS^{\cJ}$ to the $\cJ$-space associated to a map $K\to B\cJ$ gives a model of the graded Thom spectrum of the composition $K\to B\cJ\to B\cW\simeq BO\times \mZ$. 
\end{remark}

Returning to the setup of Section~\ref{sec:graded-Thom-iso}, let $W$ be a grouplike commutative 
\mbox{$\cW$-space} monoid that is cofibrant and fibrant in the positive model structure on commutative $\cW$-space monoids. Then we say that a $\cJ$-space $X$ is \emph{$W$-orientable} if the left derived functor of $\Psi_*$ takes $X$ to a $W$-orientable $\cW$-space in the sense of Definition~\ref{def:W-orientation}. For use in \cite{RSS_LogTHH-II} and elsewhere, we wish to formulate a Thom isomorphism for commutative $\cJ$-space monoids entirely in terms of symmetric spectra. 
First we again fix a map of commutative orthogonal ring spectra $\mS^{\cW}[W_{\{0\}}]\to R$ making R a cofibrant $\mS^{\cW}[W_{\{0\}}]$-algebra. Then let $R^{\mathrm{cof}} \to \Psi^{*}(R)$ be a cofibrant replacement of the
underlying commutative symmetric ring spectrum of $R$, and let us denote the underlying commutative symmetric ring spectrum of $RP$ also simply by $RP$.

\begin{theorem}\label{thm:mult-thom-iso-CSJ}
  Given a diagram $M \xleftarrow{\sim} M' \to \Psi^*(W)$ of commutative $\cJ$-space monoids in which $M$ is $\mS^{\cJ}$-good, there is a natural chain of stable equivalences
  \[
  R^{\mathrm{cof}} \sm \mS^{\cJ}[M] \simeq RP \triangle M_{h\cJ}
  \]
  of $E_{\infty}$ symmetric ring spectra. 
 \end{theorem}
\begin{proof}
Applying first the bar resolution and then $\Psi_*$ to the diagram in the theorem, we get an $E_{\infty}$ orientation of $\Psi_*(\overline M)$ of the form  $\Psi_*(\overline M)\xl{\sim} \Psi_*(\overline {M'}) \xr{\sim} W$ where again $E_{\infty}$ refers to actions by the Barratt--Eccles operad. Here we observe that $\Psi_*(\overline M)$ can be identified with the homotopy left Kan extension $\Psi_*^h(M)$. Using Corollary~\ref{cor:D-multiplicative-Thom-iso} we get a chain of natural stable equivalences 
\[
R\wedge\mS^{\cW}[\Psi_*(\overline M)]\simeq RP\triangle \Psi_*(\overline M)_{h\cW}
\]
of $E_{\infty}$ orthogonal ring spectra. On the one hand,  the canonical weak homotopy equivalence $\Psi_*^h(M)_{h\cW}\to M_{h\cJ}$ from \cite[Theorem 5.5]{Hollender-V_modules} shows that the underlying symmetric spectrum of the right hand side of the above equivalence is stably equivalent to the right hand side of the stable equivalence in the theorem. On the other hand, the fact that the lower horizontal $(\Psi_*,\Psi^*)$ adjunction in \eqref{eq:orth-sym-adjunctions} is a Quillen equivalence shows that there is a stable equivalence
\[
\Psi_*(R^{\cof}\sm \mS^{\cJ}[\overline M]) \iso \Psi_*(R^{\cof})\sm \mS^{\cW}[\Psi_*(\overline M)] \xrightarrow{\sim} R \sm \mS^{\cW}[\Psi_*(\overline M)]
\]
and hence that the underlying symmetric spectrum of the left hand side of the above equivalence is stably equivalent to 
$R^{\cof}\sm \mS^{\cJ}[\overline M]$. Finally, the latter is stably equivalent to $R^{\cof}\sm \mS^{\cJ}[M]$ by the $\mS^{\cJ}$-goodness assumption on $M$ and the $\cJ$-space analogue of Proposition~\ref{cor:bar-replacement-goodness-criterion}.
\end{proof}

Finally, we consider an application of Theorem~\ref{thm:mult-thom-iso-CSJ} that plays a prominent role in 
\cite{RSS_LogTHH-II}. The work on log structures 
in that paper is carried out in the setting of 
$\cJ$-spaces and symmetric spectra of simplicial sets as opposed to topological spaces. We shall use the Quillen equivalences given by geometric realization and the singular complex functor to relate the simplicial and topological versions of the theory, 
cf.\ \cite{Sagave-S_diagram}*{Remarks~6.7 and 9.7}. 

Given an object $(\bld{n_1},\bld{n_2})$ of $\cJ$, we adopt the notation from \cite{RSS_LogTHH-I}*{Example~2.4} and write 
$\mathbb C\langle \bld{n_1},\bld{n_2}\rangle$ for the free commutative
$\cJ$-space monoid on a generator in bidegree $(\bld{n_1},\bld{n_2})$. Working in the simplicial setting, the first author has constructed a group completion in the form of a map of commutative $\cJ$-space monoids 
$\mathbb C\langle \bld{n_1},\bld{n_2}\rangle\to \mathbb C\langle \bld{n_1},\bld{n_2}\rangle^{\gp}$
with the universal property that given a grouplike and positive fibrant commutative $\cJ$-space monoid $N$, any map of commutative $\cJ$-space monoids $\mathbb C\langle \bld{n_1},\bld{n_2}\rangle\to N$ extends to $\mathbb C\langle \bld{n_1},\bld{n_2}\rangle^{\gp}$, see \cite{Sagave_spectra-of-units}.

\begin{proposition}
Let $(\bld{n_1},\bld{n_2})$ be an object of $\cJ$ with $n_1\geq 1$ and $n_2-n_1$ even. Given a map of commutative $\cJ$-space monoids $M\to \mathbb C\langle \bld{n_1},\bld{n_2}\rangle^{\gp}$ in the simplicial setting, there is a chain of natural stable equivalences 
\[
H\mZ\wedge \mS^{\cJ}[M]\simeq H\mZ P_{\even}\triangle M_{h\cJ}
\] 
of $E_{\infty}$ symmetric ring spectra. 
\end{proposition}
In this statement, the symmetric spectra in question may be interpreted either in the simplicial setting or, by passing to the geometric realization of $M$, in the topological setting. 
 
 \begin{proof}
 We apply Theorem~\ref{thm:mult-thom-iso-CSJ} and set $W$ equal to $\Tilde V_{\mathrm{ev}}^{\cof}$ and $R$ equal to $H\mZ$ as in Section~\ref{subsec:orientability-stiefel-cW-sp}. Since $\mathbb C\langle \bld{n_1},\bld{n_2}\rangle^{\gp}$ is defined in the simplicial setting, we first choose an element of $\Tilde V_{\mathrm{ev}}^{\cof}$ in bidegree $(n_1,n_2)$ and consider the associated map of commutative $\cJ$-space monoids $\mathbb C\langle \bld{n_1},\bld{n_2}\rangle\to \mathrm{sing}(\Psi^*(\Tilde V_{\mathrm{ev}}^{\cof}))$ where $\mathrm{sing}(-)$ denotes the singular complex functor. Checking from the definitions, we find that $\pi_0\Psi^*(\Tilde V_{\mathrm{ev}})_{h\cJ}\cong2\mZ$, so
 $\mathrm{sing}(\Psi^*(\Tilde V_{\mathrm{ev}}^{\cof}))$ is grouplike and we can use the universal property of the group completion to fix an extension $\mathbb C\langle \bld{n_1},\bld{n_2}\rangle^{\gp}\to \mathrm{sing}(\Psi^*(\Tilde V_{\mathrm{ev}}^{\cof}))$. The adjoint map defines a $\Psi^*(\Tilde V_{\mathrm{ev}}^{\cof})$-orientation of the geometric realization of $\mathbb C\langle \bld{n_1},\bld{n_2}\rangle^{\gp}$ and hence an orientation of the geometric realization of $M$ by precomposition. The topological version of the statement now follows since $\mathbb C\langle \bld{n_1},\bld{n_2}\rangle^{\gp}$ being cofibrant implies that $M$ is $\cJ$-good, cf.\  \cite{RSS_LogTHH-I}*{Corollary~8.5 and Lemma~8.6}. To get the simplicial version of the statement we apply the singular complex functor. 
 \end{proof}
 
\subsection{Logarithmic topological Hochschild homology} \label{subsec:logTHH}
 The logarithmic topological Hochschild homology (in short log THH) introduced in \cites{RSS_LogTHH-I, RSS_LogTHH-II}  is an extension of the ordinary topological Hochschild homology $\THH(R)$ of a commutative symmetric ring spectrum $R$ to a theory that takes commutative symmetric ring spectra equipped with pre-log structures as input. Given a pre-log structure defined by a commutative $\cJ$-space monoid $M$ and a map of commutative $\cJ$-space monoids $M\to \Omega^{\cJ}(R)$, the associated log THH is the commutative symmetric ring spectrum
 \[
 \THH(R,M)=\THH(R)\wedge_{\mS^{\cJ}[B^{\cy}(M)]}\mS^{\cJ}[B^{\rep}(M)]
 \]
 where $B^{\cy}(M)$ denoted the cyclic bar construction on $M$ and $B^{\rep}(M)$ is a certain extension known as the \emph{replete bar construction}, see \cite{RSS_LogTHH-I}*{Section~3} for details. Much of the focus in \cite{RSS_LogTHH-I} is on analyzing this construction in the case where $M$ is \emph{repetitive} in the sense that $M$ is equivalent to the non-negative part $N_{\geq0}$ (the union of the components $N_{\{n\}}$ with $n\geq0$) for a grouplike commutative $\cJ$-space monoid $N$ that is not concentrated in $\cJ$-degree $0$. In this case $B^{\rep}(M)$ can be identified with the non-negative part $B^{\cy}(N)_{\geq0}$ of the cyclic bar construction on $N$. This is used in 
 \cite{RSS_LogTHH-I}*{Section~6} to set up a homotopy cofiber sequence relating $\THH(R)$ and $\THH(R,M)$ when $M$ is repetitive. In the special case $R=\mS^{\cJ}[M]$ with the canonical pre-log structure of $M$, we can identify $\THH(\mS^{\cJ}[M])$ with $\mS^{\cJ}[B^{\cy}(M)]$ and $\THH(\mS^{\cJ}[M],M)$ with $\mS^{\cJ}[B^{\cy}(N)_{\geq0}]$. 
 
While we do not intend to set up a complete theory of log THH for pre-log structures on commutative orthogonal ring spectra in the present paper, we do want to point out that there is an analogous notion of a repetitive commutative $\cW$-space monoid $M$: Given a grouplike and cofibrant commutative $\cW$-space monoid $N$ that is not concentrated in $\cW$-degree $0$, we set $M=N_{\geq0}$, identify $\THH(\mS^{\cW}[M])$ with $\mS^{\cW}[B^{\cy}(M)]$, and define 
 $\THH(\mS^{\cW}[M],M)$ to be $\mS^{\cW}[B^{\cy}(N)_{\geq0}]$. The arguments in \cite{RSS_LogTHH-I}*{Section~6} (see also \cite{RSS_LogTHH-II}*{Remark~3.4}) then give a homotopy cofiber sequence
 \[
\THH(\mS^{\cW}[M_{\{0\}}]) \to \THH(\mS^{\cW}[M])\to \THH(\mS^{\cW}[M],M).
 \]
 Notice that for this construction it is not even necessary that $N$ be strictly commutative; for instance, the construction also applies to the $E_{\infty}$ $\cW$-space monoid $V_F$ associated to a strict symmetric monoidal functor $F\colon \cK\to\cW$ as in Section~\ref{subsec:Thom-from-continous} provided that $B\cK$ is grouplike. 
\begin{example}\label{ex:logTHHMUP-sequence}
Let $\cU^{-1}\cU\to\cW$ be the strict symmetric monoidal functor considered in Example~\ref{ex:MUP} and let us write $V^{\cU}$ for the associated $E_{\infty}$ $\cW$-space monoid. Then $\mS^{\cW}[V^{\cU}]$ is a model of the periodic complex cobordism spectrum $\MUP$ and $\mS^{\cW}[V^{\cU}_{\geq 0}]$ is a model of $\MUP_{\geq0}$. In this case we get a homotopy cofiber sequence
\[
\THH(\MU) \to \THH(\MUP_{\geq0})\to \THH(\MUP_{\geq0},V^{\cU}_{\geq 0})
\] 
and the $\cW$-space analogue of the analysis in \cite{RSS_LogTHH-II}*{Section~2} shows that the terms can be identified as follows:
\[
\begin{aligned}
&\THH(\MU)\simeq \MU\wedge SU_+\\
&\THH(\MUP_{\geq0})\simeq \MU\wedge SU_+\vee \MUP_{>0}\wedge U_+\\
&\THH(\MUP_{\geq0},V^{\cU}_{\geq 0})\simeq \MUP_{\geq0}\wedge U_+
\end{aligned}
\]
We observe that the multiplicative structure of the log THH term is somewhat simpler than that of $\THH(\MUP_{\geq0})$. In comparison $\THH(\MUP)\simeq \MUP\wedge U_+$.
\end{example}

\appendix
\section{Criteria for \texorpdfstring{$\mS^{\cW}$}{S^W}-goodness}\label{app-sec:SW-goodness}
In this appendix we establish some convenient criteria which ensure that a \mbox{$\cW$-space} is $\mS^{\cW}$-good in the sense of Definition~\ref{def:SW-good}. A similar analysis for $\cJ$-spaces is carried out in~\cite[Section~8]{RSS_LogTHH-I}.

Consider in general a compact Lie group $G$ and let us write $\Top^{G}$ for the category of left $G$-spaces. It is well-known that this category admits a \emph{coarse} (or ``naive'') model structure in which a map is a weak equivalence or fibration if and only if the underlying map of spaces is. We say that a map of 
$G$-spaces is a \emph{$G$-cofibration} if it is a cofibration in the coarse model structure and a $G$-space is 
\emph{$G$-cofibrant} if it is cofibrant in this model structure. It is clear that the orbit space/trivial action adjunction $\Top^{G}\rightleftarrows \Top$ is a Quillen adjunction when $\Top^{G}$ is equipped with the coarse model structure. This implies in particular that a weak equivalence of 
$G$-spaces $X\to Y$ induces a weak equivalence of the orbit spaces $X/G\to Y/G$ provided that $X$ and $Y$ are 
$G$-cofibrant.  

There is also a \emph{fine} model structure on $\Top^G$ in which a map is a weak equivalence (respectively a fibration) if and only if the induced map of $H$-fixed points is a weak equivalence (respectively a fibration) for all closed subgroups $H$ in $G$ (this is also known as the ``genuine'' model structure). This model structure has more cofibrations than the coarse model structure and in particular any $G$-equivariant CW complex is cofibrant. 
The next lemma is the $G$-version of the analogous result for finite groups proved in \cite[Lemma~12.10]{Sagave-S_diagram}. 

\begin{lemma}\label{lem:coarse-pushout-product}
Let $f\colon X_1\to X_2$ and $g\colon Y_1\to Y_2$ be respectively a cofibration in the coarse and the fine model structure on $\Top^G$. Then the pushout-product
\[
f\Box g\colon X_1\times Y_2\cup_{X_1\times Y_1}X_2\times Y_1\to X_2\times Y_2
\]
is also a $G$-cofibration. \qed
\end{lemma}
As a consequence of the lemma we see that if $X$ is $G$-cofibrant and $Y_1\to Y_2$ is a cofibration in the fine model structure, then 
$X\times Y_1\to X\times Y_2$ is a $G$-cofibration (since $\emptyset\to X$ is a $G$-cofibration).

\begin{definition}\label{def:O-cof-second-var}
A $\cW$-space $X$ is \emph{$O$-cofibrant in the second variable} if for every object $(n_1,n_2)$, the ($O(n_1)\times O(n_2)$)-space $X(n_1,n_2)$ becomes $O(n_2)$-cofibrant when restricting the action to $O(n_2)$.
\end{definition}
\begin{lemma}\label{lem:O-cofibrant-level-equivalence}
Let $X\to Y$ be a level equivalence of $\cW$-spaces that are $O$-cofibrant in the second variable. Then $\mS^{\cW}[X]\to \mS^{\cW}[Y]$ is a level equivalence.
\end{lemma}

\begin{proof}
By the explicit description of $\mS^{\cW}$ in Proposition~\ref{prop:graded-SW-formula}, the statement in the lemma is equivalent to the maps 
\[
X(n_1,n_2)_+\wedge_{O(n_2)}S^{n_2}\to Y(n_1,n_2)_+\wedge_{O(n_2)}S^{n_2}\
\]
being weak homotopy equivalences for all $(n_1,n_2)$. We know from 
Lemma~\ref{lem:coarse-pushout-product} that 
$X(n_1,n_2)\times S^{n_2}$ and $Y(n_1,n_2)\times S^{n_2}$ are $O(n_2)$-cofibrant and hence that the map of orbit spaces
\[
X(n_1,n_2)\times_{O(n_2)}S^{n_2}\to Y(n_1,n_2)\times_{O(n_2)}S^{n_2}
\]
is a weak homotopy equivalence. Since the inclusions of 
$X(n_1,n_2)/O(n_2)$ and $Y(n_1,n_2)/O(n_2)$ specified by the base point of $S^{n_2}$ 
are cofibrations by Lemma~\ref{lem:coarse-pushout-product}, we conclude that also the induced map of quotient spaces is a weak homotopy equivalence.
\end{proof}

\begin{lemma}\label{lem:W-cofibrant-O-cofibrant}
Every cofibrant $\cW$-space is $O$-cofibrant in the second variable. 
\end{lemma}
\begin{proof}
It is clear from the definition that $\cW((d_1,d_2),(n_1,n_2))$ is a smooth manifold with free $O(n_2)$-action and therefore a free $O(n_2)$-CW complex. This implies that in the appropriate sense, the generating cofibrations for the $\cW$-model structure on $\Top^{\cW}$ are $O$-cofibrations in the second variable.  The statement of the lemma follows from this since in general a cofibrant $\cW$-space is a retract of a cell complex constructed from the generating cofibrations. 
\end{proof}
\begin{proposition}\label{prop:O-cofibrant-SW-good}
Let $X$ be a $\cW$-space that is $O$-cofibrant in the second variable. Then $X$ is $\mS^{\cW}$-good.
\end{proposition}
\begin{proof}
Let $Y$ be a cofibrant $\cW$-space and $Y\to X$ an acyclic fibration (that is, a cofibrant replacement of $X$). Then $Y\to X$ is a level equivalence and hence
$\mS^{\cW}[Y]\to \mS^{\cW}[X]$ is a level equivalence by Lemmas \ref{lem:O-cofibrant-level-equivalence} and \ref{lem:W-cofibrant-O-cofibrant}.  
\end{proof}

\begin{lemma}\label{lem:commutative-coarse-cofibrant}
Let $M$ be a commutative $\cW$-space monoid that is cofibrant in the model structure from Theorem~\ref{thm:positive-model-structure}. Then $M$ is $O$-cofibrant in the second variable.  
\end{lemma}
\begin{proof}
We may assume without loss of generality that $M$ is a cell complex constructed from the generating cofibrations. This means that there exists a $\lambda$-sequence 
$\{M_{\alpha}\colon \alpha<\lambda\}$ of commutative $\cW$-space monoids (for some ordinal 
$\lambda$) such that $M_0$ is the unit $U^{\cW}$ for the monoidal structure, 
$\colim_{\alpha<\lambda}M_{\alpha}=M$, and $M_{\alpha}\to M_{\alpha+1}$ is obtained by cobase change from a map of the form $\mathbb C(X_{\alpha})\to \mathbb C(Y_{\alpha})$, where $X_{\alpha}\to Y_{\alpha}$ is a generating cofibration for the positive model structure on $\Top^{\cW}$ and $\mathbb C$ denotes the monad associated to the commutative operad (see Section~\ref{subsec:mult-orientations}). 
 Here the colimit of the $M_{\alpha}$'s can be calculated in the underlying category of \mbox{$\cW$-spaces}, cf.\ ~ \cite[Lemma~9.2]{Sagave-S_diagram}.  
Since $U^{\cW}=\cW((0,0),-)$ is $O$-cofibrant in the second variable as observed in the proof of Lemma~\ref{lem:W-cofibrant-O-cofibrant}, it suffices to prove that the $(n_1,n_2)$th level of $M_{\alpha}\to M_{\alpha+1}$ is an $O(n_2)$-cofibration in the second variable for all $\alpha$. In order to set up an inductive argument we in fact prove the following stronger statement: For every object $(d_1,d_2)$ of $\cW$ and every subgroup $G$ of $\Sigma_{d_1}\times\Sigma_{d_2}$ with the property that the projection to $\Sigma_{d_1}$ is injective, the $(n_1,n_2)$th level of the map
\[
M_{\alpha}\boxtimes F_{(d_1,d_2)}^{\cW}(*)\to
M_{\alpha+1}\boxtimes F_{(d_1,d_2)}^{\cW}(*)
\] 
is an ($O(n_2)\times G$)-cofibration. (Here we implicitly convert the canonical right action of 
$\Sigma_{d_1}\times \Sigma_{d_2}$ on $F_{(d_1,d_2)}^{\cW}(*)$ to a left action by letting a group element act through its inverse.) 
To start we observe that with $(d_1,d_2)$ and $G$ as above, the $(n_1,n_2)$th level of $F_{(d_1,d_2)}^{\cW}(*)$ is ($O(n_2)\times G$)-cofibrant.
Now assume by induction that the above statement holds for all $\alpha<\beta$. Suppose that the generating cofibration $X_{\beta}\to Y_{\beta}$ has the form 
$F_{(m_1,m_2)}^{\cW}(j)$ with $m_1\geq 1$ and $j$ a generating cofibration for 
$\Top$. For the induction step we use the filtration from 
\cite[Proposition~10.1]{Sagave-S_diagram} to filter the map $M_{\beta}\to M_{\beta+1}$ by a sequence of $\cW$-spaces
\[
M_{\beta}=F_0(M_{\beta+1})\to\dots\to F_{i-1}(M_{\beta+1})\to
F_i(M_{\beta+1})\to\dots
\]
such that $\colim F_i(M_{\beta+1})=M_{\beta+1}$ and $F_{i-1}(M_{\beta+1})\to
F_i(M_{\beta+1})$ is obtained by cobase change from the map $M_{\beta}\boxtimes
F_{(m_1,m_2)^{\oplus i}}^{\cW}(j^{\Box i})/\Sigma_i$, where $j^{\Box i}$ denotes the $i$-fold iterated pushout-product of $j$ (see \cite[Section~A.6]{Sagave-S_diagram}). With 
$(d_1,d_2)$ and $G$ as in the induction statement this in turn gives rise to a filtration of the induced map $M_{\beta}\boxtimes F_{(d_1,d_2)}^{\cW}(*)\to M_{\beta+1}\boxtimes F_{(d_1,d_2)}^{\cW}(*)$ by the $\cW$-spaces $F_i(M_{\beta+1})\boxtimes 
F_{(d_1,d_2)}^{\cW}(*)$ such that the maps in the filtration are obtained by cobase change from the maps $M_{\beta}\boxtimes
F_{(a_1,a_2)}^{\cW}(j^{\Box i})/\Sigma_i$ where $(a_1,a_2)=(m_1,m_2)^{\oplus i}\oplus(d_1,d_2)$. Hence it suffices to check that the latter is an 
$(O(n_2)\times G)$-cofibration at level $(n_1,n_2)$. To this end we observe that the composition of the homomorphism
\[
\Sigma_i\times G\to \Sigma_{m_1^{\oplus i}}\times \Sigma_{m_2^{\oplus i}}\times \Sigma_{d_1}\times \Sigma_{d_2} \to \Sigma_{m_1^{\oplus i}\oplus  d_1}\times 
\Sigma_{m_2^{\oplus i}\oplus d_2}
\] 
with the projection to  $\Sigma_{m_1^{\oplus i}\oplus d_1}$ is injective (since $m_1\geq 1$). 
By the induction hypothesis this implies that  $M_{\alpha}\boxtimes F_{(a_1,a_2)}^{\cW}(*)$ is $(O(n_2)\times \Sigma_i\times G)$-cofibrant at the $(n_1,n_2)$th level for $\alpha<\beta$ and therefore also for $\alpha=\beta$.
Using Lemma~\ref{lem:coarse-pushout-product} it follows that $M_{\beta}\boxtimes F_{(a_1,a_2)}^{\cW}(j^{\Box i})$ is an ($O(n_2)\times\Sigma_i\times G$)-cofibration at level 
$(n_1,n_2)$ and hence that the map of $\Sigma_i$-orbits is an ($O(n_2)\times G$)-cofibration. 
\end{proof}

Recall that the bar resolution $\overline X$ of a $\cW$-space $X$ is not necessarily cofibrant unless $X$ is level-wise cofibrant, cf.\ Lemma~\ref{lem:bar-resolution-cofibrant}. The next result is therefore technically convenient.

\begin{proposition}\label{prop:bar-SW-good}
The bar resolution $\overline X$ of a $\cW$-space $X$ is $\mathbb S^{\cW}$-good.
\end{proposition}
\begin{proof}
As in the proof of Proposition~\ref{prop:O-cofibrant-SW-good} we choose an acyclic fibration $Y\to X$ with $Y$ cofibrant.
Then $\overline Y$ is cofibrant by Lemma~\ref{lem:bar-resolution-cofibrant}, so it is enough to show that the map of bar resolutions $\overline Y\to\overline X$ induces a level equivalence 
$\mathbb S^{\cW}[\overline Y]\to \mathbb S^{\cW}[\overline X]$. By the explicit description in Proposition~\ref{prop:graded-SW-formula}, this is equivalent to showing that 
\[
\overline Y(n_1,n_2)_+\wedge_{O(n_2)}S^{n_2}\to \overline X(n_1,n_2)_+\wedge_{O(n_2)}S^{n_2}\
\]
is a weak homotopy equivalence for all $(n_1,n_2)$. Here the domain 
can be identified with the geometric realization of the simplicial space with $p$-simplices 
\[
\bigvee_{\bld n^0,\dots,\bld n^p}\cW(\bld n^0,(n_1,n_2))_+\wedge_{O(n_2)}S^{n_2}
\wedge\big(\cW(\bld n^1,\bld n^0)\times\dots\times\cW(\bld n^p,\bld n^{p-1})\times Y(\bld n^p)\big)_+
\] 
where $\bld n^0$,\dots, $\bld n^p$ runs over all $(p+1)$-tuples of object in $\cW$. There is a similar description of the domain with $X$ instead of $Y$ and the map in question is induced by the maps 
$Y(\bld n^p)\to X(\bld n^p)$. Clearly this is a weak equivalence in each simplicial degree. It is not difficult to check that these are good simplicial spaces in the sense that the degeneracy maps are Hurewicz cofibrations (see 
e.g.\ ~\cite[Proposition~2.5]{Lewis-when-cofibration}), hence the map of geometric realizations is also a weak equivalence. 
\end{proof}

\begin{corollary}\label{cor:bar-replacement-goodness-criterion}
A $\cW$-space $X$ is $\mathbb S^{\cW}$-good if and only if the canonical map $\overline X\to X$ induces a stable equivalence $\mathbb S^{\cW}[\overline X]\to \mathbb S^{\cW}[X]$. \qed
\end{corollary}



\end{document}